\documentclass[12pt]{amsart}
\usepackage{amssymb}
\usepackage{amsmath}
\usepackage{amsthm}
\usepackage{graphicx}
\usepackage{enumerate}
\usepackage{tikz}
\usepackage{tabularx}
\usepackage{mathdots}
\usepackage{tabmac_avi}
\usetikzlibrary{arrows.meta}
\tikzset{>={Latex[width=1.2mm,length=1.7mm]}}

\newtheorem{thm}{Theorem}[section]
\newtheorem{prop}[thm]{Proposition}
\newtheorem{cor}[thm]{Corollary}
\newtheorem{lem}[thm]{Lemma}
\newtheorem{obs}[thm]{Observation}
\newtheorem{conj}[thm]{Conjecture}

\newtheorem{quest}[thm]{Question}
\newtheorem{prob}[thm]{Problem}

\numberwithin{equation}{section}

\newcommand{\sn}{\mathfrak{S}_n}

\newcommand{\mfs}[1]{\mathfrak{S}_{#1}}

\newcommand{\zsn}{\mathbb{Z}[\sn]}
\newcommand{\qsn}{\mathbb{Q}[\sn]}

\newcommand{\zqq}{\mathbb{Z}[\qp12, \qm12]}

\newcommand{\csn}{\mathbb{C}[\sn]}

\newcommand{\qp}[2]{q^{\frac{#1}{#2}}}
\newcommand{\qm}[2]{q^{\negthinspace\Bar\,\frac{#1}{#2}}}

\newcommand{\ol}[1]{\overline{#1}}

\newcommand{\hnq}{H_n(q)}

\newcommand{\trspace}[1]{\mathcal{T}_{#1}}

\newcommand{\wtc}[2]{\widetilde{C}_{#1}(#2)}
\newcommand{\dksf}[2]{\tilde{F}_{#2}^{(#1)}}
\newcommand{\ksf}[2]{s_{#2}^{(#1)}}
\newcommand{\qdkst}[2]{\delta_q^{#2,(#1)}}
\newcommand{\dkst}[2]{\delta^{#2,(#1)}}
\newcommand{\qkst}[2]{\chi_q^{#2,(#1)}}
\newcommand{\kst}[2]{\chi^{#2,(#1)}}
\newcommand{\bolk}[3]{\mathbf{K}_{#2,#3}^{(#1)}}

\newcommand{\imm}[1]{\mathrm{Imm}_{#1}}

\newcommand{\sumsb}[1]{\sum_{\substack{#1}}}  
\newcommand{\stat}{\textsc{stat}}
\newcommand{\inv}{\mathrm{inv}}
\newcommand{\pinv}{\mathrm{inv}_{P}}
\newcommand{\rinv}{\textsc{rinv}_P}
\newcommand{\defeq}{:=} 

\newcommand{\spn}{\mathrm{span}}
\newcommand{\sgn}{\mathrm{sgn}}
\newcommand{\triv}{\mathrm{triv}}
\newcommand{\wgt}{\mathrm{wgt}}
\newcommand{\type}{\mathrm{type}}
\newcommand{\sh}{\mathrm{sh}}
\newcommand{\rp}{\mathrm{rp}}

\newcommand{\pavoiding}{$3412$-avoiding, $4231$-avoiding }
\newcommand{\avoidsp}{avoids the patterns $3412$ and $4231${}}

\newcommand{\avoidingp}{avoiding the patterns $3412$ and $4231${}}

\newcommand{\ssec}[1]{\subsection{#1}{$\negthinspace$}}

\newcommand{\tr}{{\negthickspace \top \negthickspace}}
\newcommand{\ktr}[1]{{{#1}\ntnsp\top \negthickspace}}
\newcommand{\tnsp}{\hspace{1mm}}
\newcommand{\ntnsp}{\negthinspace}
\newcommand{\ntksp}{\negthickspace}
\newcommand{\nTksp}{\negthickspace\negthickspace}

\newcommand{\bp}{\begin{prob}}
\newcommand{\ep}{\end{prob}}

\newcommand{\mat}[1]{\mathrm{Mat}_{#1 \times #1}}

\newcommand{\perm}{\mathrm{per}}
\newcommand{\ctype}{\mathrm{ctype}}
\newcommand{\frobch}{\mathrm{ch}}
\newcommand{\inc}{\mathrm{inc}}
\newcommand{\ngr}{\mathrm{ngr}}
\newcommand{\permmon}[2]{#1_{1,#2_1} \ntnsp\cdots {#1}_{n,#2_n}}

\newcommand{\ssm}{\smallsetminus}
\newcommand{\multiu}{\Cup}

\newcommand{\core}[2]{\mathfrak{c}_{#1}(#2)}

\newcommand{\circd}[1]{\raisebox{-9pt}{\textcircled{\raisebox{-.9pt}{#1}}}}

\newcommand{\pexc}{\mathrm{exc}_P}
\newcommand{\paexc}{\mathrm{aexc}_P}
\newcommand{\pdes}{\mathrm{des}_P}
\newcommand{\pasc}{\mathrm{asc}_P}
\newcommand{\exc}{\mathrm{exc}}
\newcommand{\aexc}{\mathrm{aexc}}
\newcommand{\rec}{\mathrm{rec}}
\newcommand{\des}{\mathrm{des}}
\newcommand{\iDES}{\textsc{ides}}

\newcommand{\cp}{\mathrm{cp}}
\newcommand{\mcp}{\mathcal{P}}
\newcommand{\oisp}{\mathcal{I}}
\newcommand{\upparrow}{\big \uparrow \nTksp \phantom{\uparrow}}

\newcommand{\olj}{\ol{\phantom j} \nTksp\ntnsp J}
\newcommand{\phsum}{\phantom{\sum_A^Z}}
\newcommand{\phm}{\phantom M}
\newcommand{\phn}{\phantom{ni}}

\def\hhhsp{\def\baselinestretch{0.125}\large\normalsize}

\def\ssp{\def\baselinestretch{1.0}\large\normalsize}

\begin{document}
\author{Mark Skandera}
\title{Characters and chromatic symmetric functions}

\bibliographystyle{../dart}

\date{\today}

\begin{abstract}
  Given a poset $P$, let $\inc(P)$ be its incomparability graph,
  $X_{\inc(P)}$ the corresponding chromatic symmetric function
  defined by Stanley in {\em Adv. Math.}, {\bf 111} (1995) pp.~166--194.
  Let $\omega$ be the standard involution on symmetric functions.
  Certain conditions on $P$ imply that the expansions of $X_{\inc(P)}$
  and $\omega X_{\inc(P)}$ in standard symmetric function bases
  yield coefficients which have simple combinatorial interpretations.
  By expressing coefficients of $X_{\inc(P)}$ and $\omega X_{\inc(P)}$
  as character evaluations, we obtain simple combinatorial interpretations
  for coefficients of the power sum and monomial expansions of
  $\omega X_{\inc(P)}$ which hold for
  {\em all} posets $P$.
  Consequences include new combinatorial interpretations of the permanent,
  induced trivial character immanants, and power sum immanants
  of totally nonnegative matrices,
  and of the sum of elementary coefficients in the Shareshian-Wachs
  chromatic quasisymmetric function $X_{\inc(P),q}$ when $P$ is a unit interval
  order.
\end{abstract}
  

\maketitle


\section{Introduction}\label{s:intro}

The Frobenius isomorphism
from the space $\trspace n$ of symmetric group traces
to the space $\Lambda_n$ of homogeneous degree-$n$
symmetric functions,
\begin{equation}\label{eq:Frob}
  \begin{aligned}
    \mathrm{Frob}: \trspace n &\rightarrow \Lambda_n\\
    \theta &\mapsto \frac 1{n!} \sum_{w \in \sn} \theta(w) p_{\ctype(w)},
    \end{aligned}
\end{equation}
where
$\ctype(w)$ is the cycle type of $w$,
allows one to translate statements about the representation theory
of the symmetric group $\sn$ to the language of symmetric functions.
Conversely, one may use the inverse of the Frobenius isomorphism
to study symmetric functions, such as Stanley's chromatic symmetric
functions $X_G$~\cite{StanSymm}, in terms of $\sn$-class functions.
In particular,
for $G$ the incomparability graph $\inc(P)$ of a poset $P$,
we will expand $X_G$ in the standard symmetric function bases,
and we will use
the inverse Frobenius isomorphism to interpret the resulting coefficients.
Our main tool is reminiscent of the Cauchy and
dual Cauchy identities~\cite[\S I.4]{M1}
for symmetric functions in two sets of variables,
  \begin{multline}\label{eq:cauchy}
    \prod_{i,j \geq 1} \frac1{1 - y_ix_j} =
    \sum_\lambda e_\lambda(y)f_\lambda(x) =
    \sum_\lambda h_\lambda(y)m_\lambda(x) =
    \sum_\lambda \frac{p_\lambda(y)}{z_\lambda} p_\lambda(x)\\
    =\sum_\lambda s_\lambda(y)s_\lambda(x) =
    \sum_\lambda m_\lambda(y)h_\lambda(x) =
    \sum_\lambda f_\lambda(y)e_\lambda(x),
  \end{multline}
  \begin{multline}\label{eq:dualcauchy}
    \prod_{i,j \geq 1} (1 + y_ix_j) =
    \sum_\lambda e_\lambda(y)m_\lambda(x) =
    \sum_\lambda h_\lambda(y)f_\lambda(x) =
    \sum_\lambda \frac{(-1)^{n-\ell(\lambda)}p_\lambda(y)}{z_\lambda} p_\lambda(x)\\
    =\sum_\lambda s_{\lambda^\tr}(y)s_\lambda(x) =
    \sum_\lambda m_\lambda(y)e_\lambda(x) =
    \sum_\lambda f_\lambda(y)h_\lambda(x),
  \end{multline}
  with the inverse Frobenius isomorphism applied only to the symmetric
  functions in $y$.

  In Section~\ref{s:sftrace} we present standard bases of
  the trace space of the Hecke algebra $\hnq$,
  the trace space of the symmetric group algebra $\zsn$,
  and the space $\Lambda_n$ of homogeneous degree-$n$ symmetric functions.
  We show that the expansion
  of any homogeneous degree-$n$ symmetric function 
  in any standard basis of $\Lambda_n$ yields coefficients
  which are trace evaluations.
In Section~\ref{s:chrom} we apply this result to the
standard expansions of chromatic symmetric functions of the form $X_{\inc(P)}$
and the related symmetric functions $\omega X_{\inc(P)}$.
We obtain combinatorial interpretations for resulting coefficients
of $X_{\inc(P)}$ and $\omega X_{\inc(P)}$ for {\em all} posets $P$,
thus extending previous results
which hold only for special classes of posets.
In particular, we interpret the monomial and power sum coefficients
of $\omega X_{\inc(P)}$ in Theorems~\ref{t:etalambdaP} and \ref{t:psilambdaP},
respectively.
We also obtain a new proof in Proposition~\ref{p:Kali}
of Kaliszewski's interpretation of hook-Schur coefficients.
In each case, the trace evaluations allow for very simple proofs of our results.
In Section~\ref{s:tnn} we apply our interpretations of trace evaluations
to functions of totally
nonnegative matrices, obtaining new interpretations of these.
In particular, we interpret ``hook'' irreducible character immanants
in Theorem~\ref{t:hookimm}, induced trivial character immanants in
Theorem~\ref{t:etaimm}, and power sum immanants in Theorem~\ref{t:powerimm}.
These results include two new interpretations
(\ref{eq:introdes}) -- (\ref{eq:introexc})
of the permanent of a totally nonnegative matrix and play an important
role in the evaluation of hyperoctahedral group characters at elements of
the type-$BC$ Kazhdan-Lusztig basis~\cite{SkanGCEHGC}.
These results also lead to a new expression in Section~\ref{s:quasi}
for the sum of elementary coefficients of
the Shareshian-Wachs chromatic quasisymmetric function $X_{\inc(P),q}$
when $P$ is a unit interval order.




\section{Symmetric functions and traces}\label{s:sftrace}

Let $\Lambda$ be the ring of symmetric functions in 
$x = (x_1, x_2, \dotsc)$ having integer coefficients, and
let $\Lambda_n$ be the $\mathbb Z$-submodule of homogeneous functions
of degree $n$.
This submodule has rank equal to the number of integer
partitions of $n$, the
weakly decreasing positive integer sequences
$\lambda = (\lambda_1, \dotsc, \lambda_\ell)$ satisfying
$\lambda_1 + \cdots + \lambda_\ell = n$.
The $\ell = \ell(\lambda)$ components of $\lambda$ are called its {\em parts},
and we let $|\lambda| = n$ and $\lambda \vdash n$ denote that
$\lambda$ is a partition of $n$.  Given $\lambda \vdash n$,
we define the {\em transpose} partition
$\lambda^\tr = (\lambda^\tr_1, \dotsc, \lambda^\tr_{\lambda_1})$
by
\begin{equation*}
  \lambda^\tr_i = \# \{ j \,|\, \lambda_j \geq i \}.
\end{equation*}
Sometimes it is convenient to name a partition with exponential notation,
omitting parentheses and commas, so that
$4^21^6 \defeq (4, 4, 1, 1, 1, 1, 1, 1)$.
We define a {\em composition} of $n$ to be any rearrangement
of a partition of $n$ and write $\alpha \vDash n$ to denote that $\alpha$
is a composition of $n$.
Six standard bases of $\Lambda_n$ consist of the
monomial $\{m_\lambda \,|\, \lambda \vdash n\}$,
elementary $\{e_\lambda \,|\, \lambda \vdash n\}$,
(complete) homogenous $\{h_\lambda \,|\, \lambda \vdash n\}$,
power sum $\{p_\lambda \,|\, \lambda \vdash n\}$,
Schur $\{s_\lambda \,|\, \lambda \vdash n\}$, and
forgotten $\{f_\lambda \,|\, \lambda \vdash n\}$ symmetric functions.
(See, e.g., \cite[Ch.\,7]{StanEC2} for definitions.)
An involutive automorphism $\omega: \Lambda \rightarrow \Lambda$ defined
by $\omega(e_k) = h_k$ for all $k$ acts on these bases of $\Lambda_n$ by
\begin{equation*}
  \omega(s_\lambda) = s_{\lambda^\tr}, \qquad
  \omega(m_\lambda) = f_\lambda, \qquad
  \omega(e_\lambda) = h_\lambda, \qquad
  \omega(p_\lambda) = (-1)^{n-\ell(\lambda)}p_\lambda.
\end{equation*}

Let $\hnq$ be the (type $A$)
{\em Hecke algebra}, 
generated over $\zqq$ by $T_{s_1}, \dotsc, T_{s_{n-1}}$ subject to relations
\begin{equation*}
  \begin{alignedat}2
    T_{s_i}^2 &= (q-1)T_{s_i} + q &\qquad &\text{for $i = 1,\dotsc,n-1$}, \\
    T_{s_i}T_{s_j} &= T_{s_j}T_{s_i} &\qquad &\text{for $|i-j| \geq 2$}, \\
    T_{s_i}T_{s_j}T_{s_i} &= T_{s_j}T_{s_i}T_{s_j} &\qquad &\text{for $|i-j| =1$}.
  \end{alignedat}
\end{equation*}
For each $w \in \sn$ and
$w = s_{i_1} \ntnsp \cdots s_{i_\ell}$ a reduced expression,
define the natural basis element
$T_w = T_{s_{i_1}}\ntnsp \cdots T_{s_{i_\ell}}$
(which does not depend upon the choice of a reduced expression).
(See, e.g., \cite{BBCoxeter}.)
The (modified) {\em Kazhdan-Lusztig basis} of $\hnq$
as a $\zqq$-module consists of elements
$\{ \wtc wq \,|\, w \in \sn \}$ related to the natural basis by
\begin{equation}\label{eq:kl}
  \wtc wq = \sum_{v \leq w} P_{v,w}(q) T_v,
  \end{equation}
where $\leq$ is the Bruhat order on $\sn$, and
where $\{ P_{v,w}(q) \,|\, v, w \in \sn \}$ are the 
recursively defined {\em Kazhdan-Lusztig polynomials}.
(
Our basis element $\wtc wq$
is $\smash{\qp{\ell(w)}2}$ times the basis element $C'_w$
in \cite{KLRepCH}.)
When $w$ \avoidsp~(the one-line notation $w_1 \cdots w_n$ contains no subword
$w_{i_1}w_{i_2}w_{i_3}w_{i_4}$ whose letters have values appearing
in the same relative order as $4231$ or $3412$),
each polynomial $P_{v,w}(q)$ is identically $1$.

Let $\trspace{n,q}$ be the $\zqq$-module of
$\hnq$-{\em traces},
linear functionals
$\theta_q: \hnq \rightarrow \zqq$ satisfying
$\theta_q(gh) = \theta_q(hg)$ for all $g, h \in \hnq$.
For any trace $\theta_q: T_w \mapsto a(q)$ in $\trspace{n,q}$,
the $\qp12 =1$ specialization $\theta: w \mapsto a(1)$
belongs to the space $\trspace n \defeq \trspace{n,1}$ of $\zsn$-traces
from $\zsn \rightarrow \mathbb Z$ ($\sn$-class functions).
Like the $\mathbb Z$-module $\Lambda_n$,
the trace spaces $\trspace{n,q}$ and $\trspace n$ have
dimension equal to the number of integer partitions of $n$.  
The Frobenius $\mathbb Z$-module isomorphism (\ref{eq:Frob})
and its $q$-extension,
$\mathrm{Frob}_q: \trspace {n,q} \rightarrow \Lambda_n$,
$\theta_q \mapsto \mathrm{Frob}(\theta)$,
define bijections
between standard
bases of $\Lambda$, 
$\trspace n$,
and
$\trspace{n,q}$.
Schur functions correspond to irreducible characters,
\begin{equation*}
  s_\lambda \leftrightarrow \chi^\lambda \leftrightarrow \chi_q^\lambda,
\end{equation*}
while elementary and homogeneous symmetric functions
correspond to induced sign and trivial characters,
\begin{equation*}
  \begin{gathered}
  e_\lambda \leftrightarrow \epsilon^\lambda = \sgn \upparrow^{\sn}_{\mfs \lambda}
  \leftrightarrow \epsilon_q^\lambda = \sgn_q \upparrow^{\hnq}_{H_\lambda(q)},\\
  h_\lambda \leftrightarrow \eta^\lambda = \triv \upparrow^{\sn}_{\mfs \lambda}
  \leftrightarrow \eta_q^\lambda = \triv_q \upparrow^{\hnq}_{H_\lambda(q)},
  \end{gathered}
\end{equation*}
where $\mfs \lambda$ is the Young subgroup of $\sn$ indexed by $\lambda$ and
$H_\lambda(q)$ is the corresponding parabolic subalgebra of $\hnq$.
The power sum,
monomial,
and forgotten bases of $\Lambda_n$ correspond
to
bases of $\trspace n$ ($\trspace {n,q}$)
which are not characters.
We call these the
{\em power sum}
$\{\psi^\lambda \,|\, \lambda \vdash n \}$
($\{\psi_q^\lambda \,|\, \lambda \vdash n \}$),
{\em monomial}
$\{ \phi^\lambda \,|\, \lambda \vdash n \}$
($\{ \phi_q^\lambda \,|\, \lambda \vdash n \}$),
and {\em forgotten}
$\{ \gamma^\lambda \,|\, \lambda \vdash n \}$
($\{ \gamma_q^\lambda \,|\, \lambda \vdash n \}$)
traces, respectively.
These
are
the bases related to the irreducible character bases
by the same matrices of character evaluations and
inverse Koskta numbers that relate
power sum, monomial, and forgotten symmetric functions to Schur functions,
\begin{equation}\label{eq:mforgotten}
  \begin{alignedat}{3}
    p_\lambda &= \sum_\mu \chi^\mu(\lambda) s_\mu, &\qquad
    \psi^\lambda &= \sum_\mu \chi^\mu(\lambda) \chi^\mu, &\qquad
    \psi_q^\lambda &= \sum_\mu \chi^\mu(\lambda) \chi_q^\mu,\\
    m_\lambda &= \sum_\mu K_{\lambda,\mu}^{-1} s_\mu, &\qquad
    \phi^\lambda &= \sum_\mu K_{\lambda,\mu}^{-1} \chi^\mu, &\qquad
    \phi_q^\lambda &= \sum_\mu K_{\lambda,\mu}^{-1} \chi_q^\mu,\\
    f_\lambda &= \sum_\mu K_{\lambda,\mu^\tr\;}^{-1} s_\mu, &\qquad
    \gamma^\lambda &= \sum_\mu K_{\lambda,\mu^\tr}^{-1} \chi^\mu, &\qquad
    \gamma_q^\lambda &= \sum_\mu K_{\lambda,\mu^\tr}^{-1} \chi_q^\mu,
  \end{alignedat}
\end{equation}
where $\chi^\mu(\lambda) \defeq \chi^\mu(w)$
for any $w \in \sn$ having $\ctype(w) = \lambda$.
Just as the power sum symmetric functions form a
$\mathbb Q$-basis of $\Lambda_n$, the power sum traces form
$\mathbb Q$-bases of $\trspace n$ and $\trspace {n,q}$.
The power sum traces of $\trspace n$
also have the natural definition
\begin{equation}\label{eq:psidef}
  \psi^\lambda(w) \defeq \begin{cases}
    z_\lambda &\text{if $\ctype(w) = \lambda$},\\
    0 &\text{otherwise},
  \end{cases}
\end{equation}
where $z_\lambda = \lambda_1 \cdots \lambda_\ell \alpha_1! \cdots \alpha_n!$
and $\alpha_i$ is the number of parts of $\lambda$ equal to $i$.

It can be useful to record trace evaluations
in a symmetric generating function.
In particular, for $g \in \mathbb Q(q) \otimes \hnq$,
we record induced sign character evaluations by defining
\begin{equation}\label{eq:ygdef}
  Y_q(g) \defeq \sum_{\lambda \vdash n} \epsilon_q^\lambda(g) m_\lambda
  \in \mathbb Q(q) \otimes \Lambda_n.
\end{equation}
This symmetric generating function in fact gives us all of the
standard trace evaluations.
\begin{prop}\label{p:Yexpand}
  The symmetric function $Y_q(g)$ is equal to
  \begin{equation*}
      \sum_{\lambda \vdash n} \eta_q^\lambda(g)f_\lambda
    = \sum_{\lambda \vdash n}\frac{(-1)^{n-\ell(\lambda)}\psi_q^\lambda(g)}{z_\lambda}p_\lambda
    = \sum_{\lambda \vdash n} \chi_q^{\lambda^\tr}(g)s_\lambda
    = \sum_{\lambda \vdash n} \phi_q^\lambda(g)e_\lambda
    = \sum_{\lambda \vdash n} \gamma_q^\lambda(g)h_\lambda;
  \end{equation*}
equivalently, $\omega Y_q(g)$ is equal to  
\begin{equation}\label{eq:sumepsilon}
    \sum_{\lambda \vdash n} \epsilon_q^\lambda(g)f_\lambda
  = \sum_{\lambda \vdash n} \eta_q^\lambda(g)m_\lambda
  = \sum_{\lambda \vdash n}\frac{\psi_q^\lambda(g)}{z_\lambda}p_\lambda
  = \sum_{\lambda \vdash n} \chi_q^\lambda(g)s_\lambda
  = \sum_{\lambda \vdash n} \phi_q^\lambda(g)h_\lambda
  = \sum_{\lambda \vdash n} \ntnsp \gamma_q^\lambda(g)e_\lambda.\nTksp
\end{equation}
\end{prop}
\noindent
While this follows from (\ref{eq:cauchy}) -- (\ref{eq:dualcauchy}),
we provide a short proof.
\begin{proof}
  Consider the second and fourth sums in (\ref{eq:sumepsilon}),
  in which the symmetric functions and traces satisfy
  \begin{equation}\label{eq:kostkasum}
    s_\lambda = \sum_{\mu \vdash n} K_{\lambda, \mu} m_\mu,
    \qquad
    \eta_q^\mu = \sum_{\lambda \vdash n} K_{\lambda, \mu} \chi_q^\lambda.
  \end{equation}
  Using (\ref{eq:kostkasum}) to expand the fourth sum in the
  monomial symmetric function basis, we have
     \begin{equation*}
      \sum_{\lambda \vdash n} \chi_q^\lambda(g) \sum_{\mu \vdash n} K_{\lambda,\mu} m_\mu
      = \sum_{\mu \vdash n} \sum_{\lambda \vdash n} K_{\lambda,\mu} \chi_q^\lambda(g) m_\mu
      = \sum_{\mu \vdash n} \eta_q^\mu(g) m_\mu,
     \end{equation*}
     i.e., it is equal to the second sum.
    Similarly, for each of the remaining sums
    $\sum_\lambda \theta_q^\lambda(g) t_\lambda$ in (\ref{eq:sumepsilon}),
    there is a matrix $(M_{\lambda,\mu})_{\lambda,\mu \vdash n}$
    and equations
    \begin{equation*}
    s_\lambda = \sum_{\mu \vdash n} M_{\lambda, \mu} t_\mu,
    \qquad
    \theta_q^\mu = \sum_{\lambda \vdash n} M_{\lambda, \mu} \chi_q^\lambda,
    \end{equation*}
    relating it to the fourth sum.
    In particular,
    $M_{\lambda,\mu} =
    K_{\lambda^\tr,\mu}, \chi^\lambda(\mu), K^{-1}_{\mu, \lambda}, K^{-1}_{\mu, \lambda^\tr}$,
    respectively.
    (See \cite[\S 2]{RemTrans}.)    
\end{proof}

Since each symmetric function is a quasisymmetric function, researchers
sometimes express elements of $\Lambda_n$ in terms of bases of the
$\mathbb Z$-module $\mathrm{QSym}_n$ of degree-$n$ quasisymmetric
functions.  (See \cite[\S 7.19]{StanEC2} for definitions).
The coefficients arising in such expansions
also can be viewed as trace evaluations.
In particular, let
$\{ F_{n,S} \,|\, S \subseteq [n-1] \}$ be the
fundamental quasisymmetric function basis of $\mathrm{QSym}_n$.
For any Young tableau $U$ of shape $\lambda = (\lambda_1, \dotsc, \lambda_r)$,
let $U_1, \dotsc, U_r$ denote its rows, and let
$\circ$ denote concatenation of rows.
Define the {\em inverse descent set}
of $U$ by
\begin{equation*}
  \iDES(U) = \{ i \in [n-1] \,|\, i+1 \text{ appears before } i \text{ in }
  U_r \circ \cdots \circ U_1 \}.
\end{equation*}
Now we have the following fundamental
quasisymmetric expansion of $Y_q(g)$.
\begin{cor}\label{c:fundamental}
  For $\lambda = (\lambda_1, \dotsc, \lambda_r) \vdash n$
  and $S \subseteq [n-1]$, define $b(\lambda,S)$
  to be the number of standard Young tableaux $U$ of shape $\lambda$
  with $\iDES(U) = S$.
  Then we have
  \begin{equation*}
    Y_q(g) = \nTksp \sum_{S \subseteq [n-1]} \sum_{\lambda \vdash n}
    b(\lambda,S) \chi_q^{\lambda^\tr}(g) F_{n,S}; \qquad
    \omega Y_q(g) = \nTksp \sum_{S \subseteq [n-1]} \sum_{\lambda \vdash n}
    b(\lambda,S) \chi_q^{\lambda}(g) F_{n,S}.
    \end{equation*}
\end{cor}
\begin{proof}
  By \cite[Thm.\,11]{ELWQuasiSchur}, the coefficients
  $\{c_\lambda \,|\, \lambda \vdash n \}$ and $\{ d_S \,|\, S \subseteq [n-1] \}$
  appearing in the Schur and fundamental expansions 
  \begin{equation*}
    \sum_{\lambda \vdash n} c_\lambda s_\lambda
    = \nTksp \sum_{S \subseteq [n-1]} \nTksp d_S F_{n,S}
  \end{equation*}
  of a symmetric function satisfy
  \begin{equation*}
    d_S = \sum_{\lambda \vdash n}b(\lambda,S) c_\lambda.
  \end{equation*}
  The result now follows from Proposition~\ref{p:Yexpand}.
\end{proof}

To say that the functions $\{ Y_q(g) \,|\, g \in \hnq \}$ arise often
in the study of symmetric functions would be an understatement;
essentially {\em every} element of $\mathbb Z[q] \otimes \Lambda_n$ has this form.
\begin{prop}\label{p:everysymmfn}
  Every symmetric function in $\mathbb Z[q] \otimes \Lambda_n$
  has the form $Y_q(g)$ for some element $g \in \mathbb Q(q) \otimes \hnq$.
\end{prop}
\begin{proof}
  Fix a symmetric function in $\mathbb Z[q] \otimes \Lambda_n$,
  express it in the elementary basis
  as $\sum_{\lambda \vdash n} a_\lambda e_\lambda$, and define the $\hnq$
  element
  \begin{equation*}
    g = \sum_{\mu \vdash n} \frac{a_\mu}{[\mu_1]_q! \cdots [\mu_{\ell(\mu)}]_q!} \wtc{w_\mu}q,
  \end{equation*}
  where
  \begin{equation*}
    [b]_q \defeq \begin{cases}
      1 + q + \cdots + q^{b-1} &\text{if $b \geq 1$},\\
      0 &\text{if $b = 0$};
    \end{cases}
    \qquad \qquad
    [b]_q! \defeq \begin{cases}
      [1]_q [2]_q \cdots [b]_q &\text{if $b \geq 1$},\\
      1 &\text{if $b = 0$};
    \end{cases}
  \end{equation*}
  and
  $w_\mu$ is the maximal element of the Young subgroup
  $\mfs \mu$ of $\sn$.
  By
  \cite[Prop.\,4.1]{HaimanHecke}, we have
  \begin{equation*}
    \phi^\lambda_q(\wtc{w_\mu}q) =
    \begin{cases}
      [\mu_1]_q! \cdots [\mu_{\ell(\mu)}]_q! &\text{if $\lambda = \mu$},\\
      0 &\text{otherwise}.
    \end{cases}
  \end{equation*}
  Thus $Y_q(g)$ is equal to
  \begin{equation}\label{eq:strongestposs}
    \sum_{\lambda \vdash n} \phi^\lambda_q
      \Big( \sum_{\mu \vdash n}
      \frac{a_\mu}{[\mu_1]_q! \cdots [\mu_{\ell(\mu)}]_q!} \wtc{w_\mu}q \Big) e_\lambda
      = \sum_{\lambda \vdash n} 
      \sum_{\mu \vdash n} a_\mu 
      \frac{\phi^\lambda_q(\wtc{w_\mu}q)}{[\mu_1]_q! \cdots [\mu_{\ell(\mu)}]_q!} e_\lambda
      = \sum_{\lambda \vdash n} a_\lambda e_\lambda.
  \end{equation}
\end{proof}
Of course, for $g \in \qsn$, the $\qp12 = 1$ specialization
$Y(g) \defeq Y_1(g)$ of (\ref{eq:ygdef}) satisfies the
$\qp12 = 1$ specializations of Proposition~\ref{p:Yexpand},
Corollary~\ref{c:fundamental}, and Proposition~\ref{p:everysymmfn}.




\section{Chromatic symmetric functions}\label{s:chrom}

Closely related to symmetric generating functions for $\zsn$-traces
are symmetric generating functions for graph colorings.
Define a {\em proper coloring} of a (simple undirected) graph $G = (V,E)$
to be an assignment $\kappa: V \rightarrow \{1,2,\dotsc, \}$
of colors (positive integers) to $V$
such that adjacent vertices have different colors.
For $G$ on $|V| = n$ vertices and any composition
$\alpha = (\alpha_1,\dotsc,\alpha_\ell) \vDash n$,
say that a coloring $\kappa$ of $G$ has {\em type}
$\alpha$ if $\alpha_i$ vertices have color $i$ for $i = 1,\dotsc, \ell$.
Let $c(G,\alpha)$ be the number of proper colorings of $G$ of type
$\alpha$.
Stanley \cite{StanSymm} defined the {\em chromatic symmetric function}
of $G$ to be
\begin{equation}\label{eq:chrom}
  X_G \defeq \sum_\kappa x_{\kappa(1)} \cdots x_{\kappa(n)}
  = \sum_{\lambda \vdash n} c(G,\lambda) m_\lambda,
\end{equation}
where the first sum is over all proper colorings of $G$.
By Proposition~\ref{p:everysymmfn} we see that for each graph $G$ on $n$
vertices, there exists an element $g \in \qsn$ such that $X_G = Y(g)$.
Such an element $g$ is not uniquely determined by $G$, and is not in general
easily described in terms of the structure of $G$.
On the other hand, the evaluations of traces at such elements are easily
described in terms of $G$.
\begin{obs}\label{o:yyx}
  Let $G$ be a graph on $n$ vertices
  and let $g \in \qsn$ satisfy $Y(g) = X_G$.
  Then for each trace
  $\theta  = \sum_{\lambda \vdash n} a_\lambda \epsilon^\lambda \in \trspace n$,
  we have
  $\theta(g) = \sum_{\lambda \vdash n} a_\lambda c(G,\lambda)$.
\end{obs}
\begin{proof}
  By (\ref{eq:ygdef}) and (\ref{eq:chrom}) we have
  $\epsilon^\lambda(g) = c(G,\lambda)$ for each $\lambda \vdash n$.
\end{proof}
For every trace
$\theta \in \trspace n$, Proposition~\ref{p:everysymmfn} and
Observation~\ref{o:yyx} allow us to define
\begin{equation}\label{eq:thetaG}
  \theta(G) \defeq \theta(g),
\end{equation}
where $g$ is any element in $\qsn$ satisfying $Y(g) = X_G$.
By Proposition~\ref{p:Yexpand}, we have that
$X_G = \sum_{\lambda \vdash n} \epsilon^\lambda(G) m_\lambda$ is equal to
  \begin{equation}\label{eq:sumXG}
      \sum_{\lambda \vdash n} \eta^\lambda(G)f_\lambda
    = \sum_{\lambda \vdash n}\frac{(-1)^{n-\ell(\lambda)}\psi^\lambda(G)}{z_\lambda}p_\lambda
    = \sum_{\lambda \vdash n} \chi^{\lambda^\tr}(G)s_\lambda
    = \sum_{\lambda \vdash n} \phi^\lambda(G)e_\lambda
    = \sum_{\lambda \vdash n} \gamma^\lambda(G)h_\lambda;
  \end{equation}
equivalently, $\omega X_G$ is equal to  
\begin{equation}\label{eq:sumomegaXG}
    \sum_{\lambda \vdash n} \epsilon^\lambda(G)f_\lambda
  = \sum_{\lambda \vdash n} \eta^\lambda(G)m_\lambda
  = \sum_{\lambda \vdash n}\frac{\psi^\lambda(G)}{z_\lambda}p_\lambda
  = \sum_{\lambda \vdash n} \chi^\lambda(G)s_\lambda
  = \sum_{\lambda \vdash n} \phi^\lambda(G)h_\lambda
  = \sum_{\lambda \vdash n} \gamma^\lambda(G)e_\lambda.
\end{equation}


Some conditions on graphs $G$ and traces $\theta$ imply
the numbers $\theta(G)$
to be positive, and sometimes the resulting positive numbers have
nice combinatorial interpretations, particularly when $G$ is the
incomparability graph of a poset.
(See, e.g., \cite{CHSSkanEKL}, \cite{SWachsChromQF}, \cite{StanSymm}.)
Given a poset $P$, define its {\em incomparability graph} $\inc(P)$
to be the graph having a vertex for each element of $P$ and
an edge $\{i,j\}$ for each incomparable pair of elements of $P$.
For positive integers $a$, $b$, call a poset
($\mathbf{a}+\mathbf{b}$)-{\em free}
if it has no induced subposet isomorphic to a disjoint sum of an $a$-element
chain and a $b$-element chain.
For example, the following poset $P$ is
$(\mathbf3 + \mathbf 1)$-free and
$(\mathbf2 + \mathbf 2)$-free,
and has incomparability graph $\inc(P) = G$.
\begin{equation}\label{eq:n+1}
  P = \ntnsp
\begin{tikzpicture}[scale=.7,baseline=-5]
\draw[fill] (0,.5) circle (1mm); \node at (0,1) {$4$};
\draw[fill] (0,-.5) circle (1mm); \node at (0,-1) {$1$};
\draw[fill] (.8,.5) circle (1mm); \node at (.8,1) {$5$};
\draw[fill] (.8,-.5) circle (1mm); \node at (.8,-1) {$2$};
\draw[fill] (1.6,0) circle (1mm); \node at (1.6,-.5) {$3$};
\draw[-,thick] (0,.5) -- (0,-.5);
\draw[-,thick] (.8,.5) -- (0,-.5);
\draw[-,thick] (.8,.5) -- (.8,-.5);
\end{tikzpicture}\,,
\qquad 
  (\mathbf3 + \mathbf1) =\,
  \begin{tikzpicture}[scale=.7,baseline=-5]
\draw[fill] (0,1) circle (1mm); 
\draw[fill] (0,0) circle (1mm); 
\draw[fill] (0,-1) circle (1mm); 
\draw[fill] (.8,0) circle (1mm); 
\draw[-,thick] (0,1) -- (0,0);
\draw[-,thick] (0,0) -- (0,-1);
\end{tikzpicture}\, ,
  \qquad
  (\mathbf2 + \mathbf2) =\,
\begin{tikzpicture}[scale=.7,baseline=-5]
\draw[fill] (0,.5) circle (1mm); 
\draw[fill] (0,-.5) circle (1mm); 
\draw[fill] (.8,.5) circle (1mm); 
\draw[fill] (.8,-.5) circle (1mm); 
\draw[-,thick] (0,.5) -- (0,-.5);
\draw[-,thick] (.8,.5) -- (.8,-.5);
\end{tikzpicture}\,,
\qquad 
  G = \ntnsp
\begin{tikzpicture}[scale=.7,baseline=-5]
\draw[fill] (-1.2,.5) circle (1mm); \node at (-1.2,1) {$1$};
\draw[fill] (-.4,.5) circle (1mm); \node at (-.4,1) {$2$};
\draw[fill] (.4,.5) circle (1mm); \node at (.4,1) {$4$};
\draw[fill] (0,-.5) circle (1mm); \node at (0,-1) {$3$};
\draw[fill] (1.2,.5) circle (1mm); \node at (1.2,1) {$5$};
\draw[-,thick] (-1.2,.5) -- (0,-.5);
\draw[-,thick] (-.4,.5) -- (0,-.5);
\draw[-,thick] (.4,.5) -- (0,-.5);
\draw[-,thick] (1.2,.5) -- (0,-.5);
\draw[-,thick] (-1.2,.5) -- (-.4,.5);
\draw[-,thick] (-.4,.5) -- (.4,.5);
\draw[-,thick] (.4,.5) -- (1.2,.5);
\end{tikzpicture}\,.
\end{equation}

If we cannot interpret $\theta(\inc(P))$ for {\em all} posets $P$, sometimes
we can do so when $P$ is $(\mathbf3 + \mathbf 1)$-free.
For an $n$-element poset $P$ which is
both ($\mathbf3+\mathbf1$)-free and
($\mathbf2+\mathbf2$)-free,
also called a {\em unit interval order},
a simple procedure produces an element $g \in \sn$
satisfying $Y(g) = X_{\inc(P)}$.
Explicitly,
for each element $y \in P$, compute
\begin{equation*}
  \beta(y) \defeq
  \# \{ x \in P \,|\, x \leq_P y \} - \# \{ z \in P \,|\, z \geq_P y \}
\end{equation*}
and label the poset elements $1, \dotsc, n$
so that we have
\begin{equation}\label{eq:altrespect}
  \beta(1) \leq \cdots \leq \beta(n).
  \end{equation}
Then define $w = w(P) = w_1 \cdots w_n$ by
\begin{equation}\label{eq:uioto312avoid}
  w_j = \max ( \{ i \,|\, i \not >_P j \} \ssm \{ w_1, \dotsc, w_{j-1} \} ).
\end{equation}
For example, elements of the poset $P$ in (\ref{eq:n+1}) are already labeled
to satsify (\ref{eq:altrespect}),
\begin{equation*}
  (\beta(1), \beta(2), \beta(3), \beta(4), \beta(5)) = (-2, -1, 0, 1, 2).
\end{equation*}
Thus we compute $w(P) = 34521$.
The labeling (\ref{eq:altrespect}) of $P$ is {\em natural} in the sense that
elements labeled $a_1$, $a_2$ satisfy
\begin{equation}\label{eq:natlabel}
  a_1 <_P a_2 \quad \Longrightarrow \quad a_1 < a_2 \text{ (as integers).}
\end{equation}

The map $P \mapsto w(P)$ is a bijection from $n$-element unit interval orders
to the
$\tfrac{1}{n+1}\tbinom{2n}n$ $312$-avoiding permutations in $\sn$,
and gives us the following result~\cite[Cor.\,7.5]{CHSSkanEKL}.
\begin{prop}\label{p:uioto312avoid}
  Let $P$ be an $n$-element unit interval order and $w = w(P)$ the
  corresponding $312$-avoiding permutation in $\sn$.  Then we have
  $X_{\inc(P)} = Y(\smash{\wtc w1})$.
\end{prop}

Combinatorial interpretations of numbers $\theta(\inc(P))$ often
involve structures called $P$-tableaux and statistics on these.
Define a {\em $P$-tableau of shape $\lambda \vdash |P|$} to be
a filling of a (French) Young diagram of shape $\lambda$ with the
elements of $P$, one per box.  Given such a $P$-tableau $U$,
let $U_i$ be the $i$th row (from the bottom) of $U$,
and let $U_{i,j}$ be the $j$th entry in row $i$.
If $P$-tableau $U$ consists of a single row,
we will call it a {\em $P$-permutation}.
In particular, each concatenation $U_{i_1} \circ \cdots \circ U_{i_k}$
of the rows (in any order) of a $P$-tableau $U$ is a $P$-permutation.
If the elements of a poset are $[n] \defeq \{1, \dotsc, n\}$,
we will sometimes write a $P$-permutation as an ordinary permutation
$v_1 \cdots v_n \in \sn$.
For example, (\ref{eq:tableauxposet}) shows a poset $P$, a
$P$-tableau $U$ of shape $32$, the second row $U_2$
and entry $U_{2,1}$ of $U$, and the $P$-permutation $U_2 \circ U_1$
which may be viewed as an element of $\mfs 5$.
\begin{equation}\label{eq:tableauxposet}
  P =
\begin{tikzpicture}[scale=.7,baseline=-5]
\draw[fill] (0,1) circle (1mm); \node at (-.5,1) {$5$};
\draw[fill] (0,0) circle (1mm); \node at (-.5,0) {$3$};
\draw[fill] (0,-1) circle (1mm); \node at (-.5,-1) {$1$};
\draw[fill] (.8,.5) circle (1mm); \node at (1.3,.5) {$4$};
\draw[fill] (.8,-.5) circle (1mm); \node at (1.3,-.5) {$2$};
\draw[-,thick] (0,1) -- (0,0);
\draw[-,thick] (0,0) -- (0,-1);
\draw[-,thick] (0,1) -- (.8,-.5);
\draw[-,thick] (0,-1) -- (.8,.5);
\draw[-,thick] (.8,.5) -- (.8,-.5);
\end{tikzpicture},
\qquad 
U = {\tableau[scY]{5,4|1,2,3}}\,,
\qquad \,
\begin{gathered}
\underset{\phantom m}{U_2} = {\tableau[scY]{5,4}}\,,\\
U_{2,1} = 5,
\end{gathered}
\qquad \,
\begin{aligned}
\underset{\phantom m}{U_2 \circ U_1} &= {\tableau[scY]{5,4,1,2,3}}\,,\\
&= 54123.
\end{aligned}
\end{equation}
The statistics we apply to $P$-tableaux are
$P$-analogs of traditional
permutation statistics.
Call a position $(i,j)$ in $U$ a {\em $P$-descent}
if $U_{i,j} >_P U_{i,j+1}$,
and define $\des_P(U)$ to be the number of $P$-descents in $U$.
Define $\ol U$ to be the $P$-tableau obtained from $U$ by ordering the elements
in each row from least to greatest labels.  That is, 
$\ol U_{i,j}$ is the entry of $U_i$ whose label is $j$th-smallest,
as an integer.
Call a position $(i,j)$ in $U$ a {\em $P$-excedance}
if $U_{i,j} >_P \ol U_{i,j}$,
and define $\exc_P(U)$ to be the number of $P$-excedances in $U$.
Call a position $(i,j)$ in $U$ a {\em $P$-record}
if $U_{i,1}, \dotsc, U_{i,j-1} <_P U_{i,j}$, and call the record {\em nontrivial}
if $j \neq 1$.
Define $\rec_P(U)$ to be the number of nontrivial $P$-records in $U$.
For example, we look again at the poset $P$ in (\ref{eq:tableauxposet}),
the $P$-tableaux there, and three more,
\begin{equation}\label{eq:tabx}
  \,
  T \ntnsp= {\tableau[scY]{5|2,4|1,3}}\,,\ \,
  U \ntnsp= {\tableau[scY]{5,4|1,2,3}}\,,\ \,
  V \ntnsp= {\tableau[scY]{4,5|1,3,2}}\,,\ \,
  W \ntnsp= {\tableau[scY]{5,4|3,1,2}}\,,\ \,
  U_2 \circ U_1 \ntnsp= {\tableau[scY]{5,4,1,2,3}}\,.
\end{equation}  
We have $\des_P(T) = \des_P(U) = \des_P(V) = 0$,
while $\des_P(W) = \des_P(U_2 \circ U_1) = 1$, because
$3 >_P 1$ and $4 >_P 1$.
We have that $\rec_P(U) = \rec_P(U_2 \circ U_1) = \rec_P(W) = 0$
since the first entry in each row of these tableaux
is greater than or incomparable to the remaining entries in the same row.
On the other hand, $\rec_P(T) = 2$ since $1 <_P 3$ and
$2 <_P 4$, and $\rec_P(V) = 1$ since $1 <_P 3$.
Reordering the entries of each row in the tableaux above
we obtain
\begin{equation*}
  \ol T = T, \qquad
  \ol U = \ol V = \ol W = {\tableau[scY]{4,5|1,2,3}}\,, \qquad
  \ol{U_2 \circ U_1} = {\tableau[scY]{1,2,3,4,5}}\,.
\end{equation*}
Comparing
$T$, $U$, $V$ with $\ol T$, $\ol U$, $\ol V$,
we have $\exc_P(T) = \exc_P(U) = \exc_P(V) = 0$.
Comparing $W$ and $\ol W$, we see that
position $(1,1)$ is the only $P$-excedance:
$3 >_P 1$. Thus $\exc_P(W) = 1$.
Comparing $U_2 \circ U_1$ to $\ol{U_2 \circ U_1}$, we see that
positions $(1,1)$ and $(1,2)$ are both $P$-excedances:
$5 >_P 1$, $4 >_P 2$.
Thus $\exc_P(U_2 \circ U_1) = 2$.

Using comparability in $P$ and the above statistics,
we define six classes of $P$-tableaux.
Call a $P$-tableau $U$ of shape $\lambda$
\begin{enumerate}
\item {\em $P$-descent-free} or {\em row-semistrict} if $\des_P(U) = 0$,
\item {\em column-strict} if the entries of each column satisfy
  $U_{i,j} <_P U_{i+1,j}$,
\item {\em standard} if it is column-strict and row-semistrict,
\item {\em cyclically row-semistrict} if it is row-semistrict and if
  $U_{i,\lambda_i} \not >_P U_{i,1}$ for all $i$,
\item {\em $P$-excedance-free} if $\exc_P(U) = 0$,
\item {\em $P$-record-free} if it has no nontrivial $P$-records.
\end{enumerate}
For example, we may examine the tableaux in (\ref{eq:tabx})
for these properties to obtain the table
\hhhsp
\begin{center}
\newcolumntype{R}{>{$}c<{$}}
\begin{tabularx}{136.85mm}{|R|R|R|R|R|R|}%
\hline
\phsum & \phm T\phm & \phm U\phm & \phm V\phm & \phn W\phn & U_2 \circ U_1 \\
\hline 
\phsum\mbox{row-semistrict}\phsum
& \checkmark & \checkmark & \checkmark &            &\\
\phsum\mbox{column-strict}\phsum
&            & \checkmark & \checkmark & \checkmark & \checkmark \\
\phsum\mbox{standard}\phsum
&            & \checkmark & \checkmark &            &\\
\phsum\nTksp\mbox{cyclically row-semistrict}\phsum\nTksp
&            &            & \checkmark &            &\\
\phsum\mbox{$P$-excedance-free}\phsum
& \checkmark & \checkmark & \checkmark &            &\\
\phsum\mbox{$P$-record-free}\phsum
&            & \checkmark &            & \checkmark & \checkmark\\
\hline
\end{tabularx}\, ,
\end{center}
\ssp
where
the row-semistrict tableaux $T$ and $U$ fail to be
cyclically row-semistrict because their first rows begin with $1$ and
end with $3 >_P 1$.

\ssec{Induced sign characters / monomial coefficients of $X_{\inc(P)}$}
By definition, the induced sign characters satisfy
\begin{equation}\label{eq:epsilonlambdaP}
  \begin{aligned}
    \epsilon^\lambda(G) &= c(G,\lambda),\\
    \epsilon^\lambda(\inc(P)) &=
    \# \text{ column-strict $P$-tableaux of shape $\lambda^\tr$}
    \end{aligned}
\end{equation}
for all graphs $G = (V,E)$ and posets $P$.
Since $\epsilon^n(G) = 1$ if $G$ has no edges and is $0$ otherwise, we can
easily express $\epsilon^\lambda(G)$ in terms of subgraphs of $G$.
For $J \subseteq V = [n]$,
let \tnsp $\olj = [n] \ssm J$, and define
\begin{equation*}
  \begin{aligned}
    G_{\ntnsp J} &= \text{subgraph of $G$ induced by vertices $J$},\\
    P_{\ntnsp J} &= \text{subposet of $P$ induced by elements $J$}.
    \end{aligned}
\end{equation*}
Given $\alpha = (\alpha_1,\dotsc,\alpha_r) \vDash n$,
call a sequence $(I_1,\dotsc,I_r)$ of subsets of $[n]$ an
{\em ordered set partition of $[n]$ of type $\alpha$} if we have
\begin{enumerate}
\item $|I_i| = \alpha_i$ for $i = 1,\dotsc, r$,
\item $I_i \cap I_j = \emptyset$ for $i \neq j$,
\item $I_1 \cup \cdots \cup I_r = [n]$.
  \end{enumerate}
Using (\ref{eq:epsilonlambdaP}) and the language of ordered set partitions,
we can decompose some trace evaluations $\theta(G)$
as follows.
\begin{lem}\label{l:efactor}
  Let $G$ be a graph on $n$ vertices.
  \begin{enumerate}
    \item If $\lambda = (\lambda_1,\dotsc,\lambda_r) \vdash n$ is the weakly
  decreasing rearrangement of the parts of $\mu \vdash k$ and $\nu \vdash n-k$,
  then we have
\begin{equation}\label{eq:epsilonfactor}
  \epsilon^\lambda(G)
  = \nTksp \sum_{(I_1,\dotsc,I_r)} \nTksp
  \epsilon^{\lambda_1}(G_{\ntnsp I_1}) \cdots \epsilon^{\lambda_r}(G_{\ntnsp I_r})
  = \ntnsp \sum_{J \subseteq [n]} \ntnsp
  \epsilon^\mu(G_{\ntnsp J}) \epsilon^\nu(G_{\olj}),
\end{equation}
where the first sum is over ordered set partitions of $[n]$
of type $\lambda$.
\item Let
symmetric functions
$t_1 \in \Lambda_k$,
$t_2 \in \Lambda_{n-k}$,
$t_1t_2 \in \Lambda_n$
correspond by the Frobenius isomorphism to  
  traces $\theta_1 \in \trspace k$,
  $\theta_2 \in \trspace{n-k}$,
  $\theta = \theta_1 \otimes \theta_2 \upparrow_{\mfs k \times \mfs{n-k}}^{\sn}
  \in \trspace n$.
Then we have
\begin{equation}\label{eq:thetafactor}
  \theta(G) = \ntnsp\sumsb{J \subseteq [n]\\|J|=k}\ntnsp
  \theta_1(G_{\ntnsp J})\theta_2(G_{\olj}).
\end{equation}
\end{enumerate}
\end{lem}
\begin{proof}
  (1) $c(G,\lambda)$ equals the number of ordered set partitions
  $(I_1,\dotsc,I_r)$ of $[n]$ of type $\lambda$ such that
  $G_{I_j}$ is an independent set for all $j$,
  and also the
  number of ordered set partitions
  of $[n]$ of type $(\mu_1, \dotsc, \mu_{\ell(\mu)}, \nu_1, \dotsc, \nu_{\ell(\nu)})$
  having the same property.

  (2) Express $t_1$, $t_2$ in the elementary bases of
  $\Lambda_k$, $\Lambda_{n-k}$ as
  \begin{equation}\label{eq:texp}
    t_1 = \sum_{\mu \vdash k} a_\mu e_\mu, \qquad
    t_2 = \sum_{\nu \vdash k} b_\nu e_\nu,
  \end{equation}
  and let $\lambda(\mu, \nu) \vdash n$
  be the weakly decreasing rearrangement of the parts of $\mu$ and $\nu$.
  Then we have
  \begin{equation*}
    t_1 t_2 = \ntksp\sumsb{\mu \vdash k\\\nu \vdash n-k}\ntksp
    a_\mu b_\nu e_{\lambda(\mu,\nu)},
    \qquad
    \theta(G) =
    \ntksp\sumsb{\mu \vdash k\\\nu \vdash n-k}\ntksp
    a_\mu b_\nu \epsilon^{\lambda(\mu,\nu)}(G).
  \end{equation*}
 By (\ref{eq:epsilonfactor}) and (\ref{eq:texp}), $\theta(G)$ equals
  \begin{equation*}
     \sumsb{\mu \vdash k\\\nu \vdash n-k} \ntnsp  
     a_\mu b_\nu \ntnsp \sumsb{J \subseteq [n]\\|J|=k} \ntnsp
     \epsilon^\mu(G_{\ntnsp J})\epsilon^\nu(G_{\olj})    
    = \ntnsp \sumsb{J \subseteq [n]\\|J|=k}
        \sum_{\mu \vdash k} a_\mu \epsilon^\mu(G_{\ntnsp J})
    \ntnsp\sum_{\nu \vdash n-k}\ntnsp b_\nu \epsilon^\nu(G_{\olj})\\
     = \ntnsp\sumsb{J\subseteq[n]\\|J|=k}\ntnsp \theta_1(G_{\ntnsp J}) \theta_2(G_{\olj}).
  \end{equation*}
\end{proof}

We will use this fact to prove similar formulas
for induced trivial characters
and power sum traces.

\ssec{Irreducible characters / Schur coefficients of $X_{\inc(P)}$}
While $\chi^\lambda(\inc(P))$ is negative for some posets $P$,
Stanley and Stembridge~\cite[Conj.\,5.1]{StanStemIJT}
conjectured it to be nonnegative
for $(\mathbf3 + \mathbf 1$)-free posets $P$.
Gasharov~\cite{GashInc} proved this by showing that for these posets
we have
\begin{equation}\label{eq:Gash}
  \chi^\lambda(\inc(P)) = 
  \# \text{ standard $P$-tableaux of shape $\lambda$}.
  \end{equation}
Kaliszewski~\cite[Prop.\ 4.3]{KaliHook} extended this result to
all posets $P$ 
when $\lambda$ is a hook shape.
We give an alternate proof of this fact using
(\ref{eq:epsilonlambdaP})
and the inverse Kostka numbers,
which satisfy
\begin{equation}\label{eq:epsilontochi}
  \chi^\mu = \sum_{\lambda \vdash n} K^{-1}_{\lambda,\mu^\tr\;} \epsilon^\lambda.
\end{equation}
For partitions $\lambda$, $\mu$ with $|\mu| \leq |\lambda| = n$
and $\mu_i \leq \lambda_i$ for all $i$, define a 
{\em (skew) Young diagram of shape $\lambda/\mu$} to be the diagram
obtained from a Young diagram of shape $\lambda$ by removing the $\mu_i$
leftmost boxes in row $i$ for all $i$.  Call a Young diagram 
a {\em border strip} if it contains no $2 \times 2$ subdiagram of boxes.
Define a {\em special ribbon diagram} of shape $\mu \vdash n$ and type
$\lambda = (\lambda_1, \dotsc, \lambda_\ell) \vdash n$
to be a Young diagram of shape $\mu$ subdivided into 
border strips (ribbons)
of sizes $\lambda_1, \dotsc, \lambda_\ell$, each of which contains a cell
from the first row of $\mu$.  Given a special ribbon diagram $Q$,
define $\sgn(Q)$ to be $-1$ to the number of pairs of boxes in $Q$ which
are horizontally adjacent and which belong to the same ribbon.
It is known that we have 
\begin{equation}\label{eq:specribdi}
K_{\lambda, \mu^\tr}^{-1} = \sum_Q \sgn(Q),
\end{equation}
where the sum is over all 
special ribbon diagrams $Q$ of shape $\mu$ and
type $\lambda$. (See \cite[\S 2]{RemTrans}.)
For example, to expand $s_{3111}$ in the elementary basis, we draw a Young
diagram of shape $3111$ and fill it with special ribbon diagrams
\begin{equation}\label{eq:special}
  \raisebox{-6mm}{\includegraphics[height=16mm]{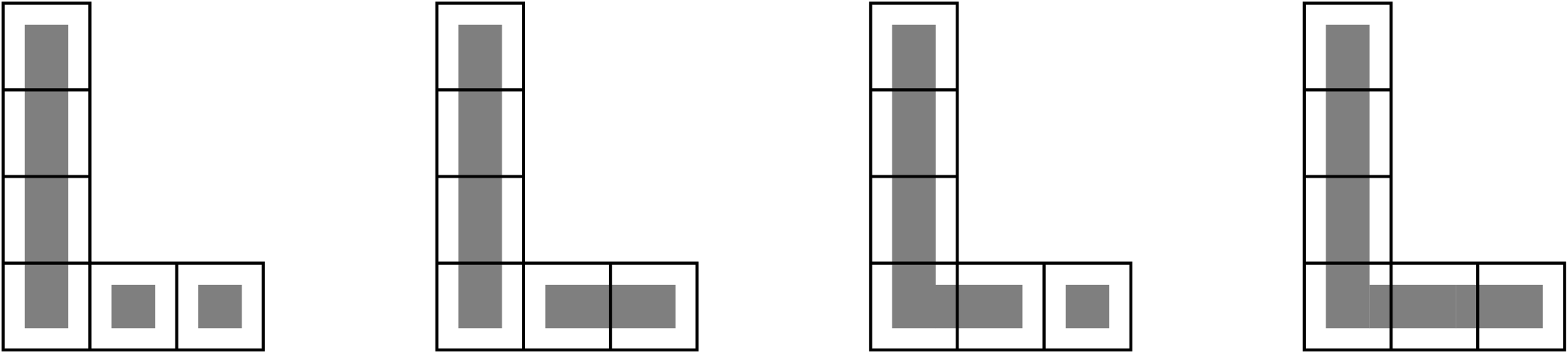}}
\end{equation}
of types $411$, $42$, $51$, $6$, respectively,
to obtain $s_{3111} = e_{411} - e_{42} - e_{51} + e_6$.

Special ribbon diagrams also relate column-strict $P$-tableaux of hook shape
$\mu = k1^{n-k} \vdash n$ to column-strict $P$-tableaux
of shape $\lambda \vdash n$ majorizing $\mu^\tr$.
To state this relationship precisely, we first observe that
subsets of $[k-1]$ correspond bijectively to special ribbon diagrams
of shape $\mu$.
Let $D$ be a Young diagram of shape $\mu$.
For each subset $S \subseteq [k-1]$, define $Q(\mu,S)$ to be
the special ribbon diagram of shape $\mu$ whose ribbons are the equivalence
classes defined by
\begin{enumerate}
\item $D_{i,1} \sim D_{i+1,1}$ for $i = 1,\dotsc,n-k$,
\item $D_{1,j} \sim D_{1, j+1}$ for all $j \in S$,
\end{enumerate}
and define $\lambda(\mu,S)$ to be the type of $Q(\mu,S)$.
For example, when $\mu = 31^3 = (3, 1, 1, 1)$
the ribbon diagrams in (\ref{eq:special})
correspond to the subsets
$\emptyset$, $\{2\}$, $\{1\}$ $\{ 1, 2\}$, respectively.
This bijection leads to another.

\begin{lem}\label{l:ribbon}
  Fix hook partition $\mu = k1^{n-k}$ and subset $S \subseteq [k-1]$,
  and define $Q = Q(\mu,S)$
  and $\lambda = \lambda(\mu,S)$ as above.
  There is a bijection between
  column-strict $P$-tableaux $U$ of shape $\mu$ satisfying
  $U_{1,j} >_P U_{1,j+1}$ for all $j \in S$,
    and column-strict $P$-tableaux of shape $\lambda^\tr$.
\end{lem}
\begin{proof}
  Let $\varphi$ be the claimed bijection.
  Given a column-strict $P$-tableau $U$ of shape $\mu$
  satisfying $U_{1,j} >_P U_{1,j+1}$ for all $j \in S$, create
  $P$-tableau $\varphi(U)$ as follows.
  \begin{enumerate}
  \item Let $D$ be a Young diagram of shape $\lambda^\tr$.
  \item Label the ribbons of $Q$ from left to right as $1, \dotsc, r$,
    and let $q_i$ be the number of boxes in ribbon $i$.
  \item Superimpose $Q$ onto $U$.
  \item For $i = 1,\dotsc, r$,
    place the elements of $U$ under ribbon $i$
    into the leftmost unused column of $D$ which contains exactly $q_i$ boxes,
    so that elements strictly increase in the column from bottom to top.    
  \end{enumerate}
  The resulting $P$-tableau is clearly column-strict of shape $\lambda^\tr$.
  To invert $\varphi$, suppose we are given a column-strict $P$-tableau $T$
  of shape $\lambda^\tr$.
  \begin{enumerate}
    \item Superimpose $Q$ onto an empty tableau $U$ of shape $\mu$.
    \item For $i = 1, \dotsc, r$, fill $U$ by placing the elements
      from column $i$ of $T$ onto the leftmost available ribbon of $Q$
      of length $\lambda_i^\tr$ boxes,
      so that elements decrease from top to bottom, and from left to right.
  \end{enumerate}
  By the definition of $Q$, the resulting $P$-tableau $U$ has entries which
  satisfy the required inequalities,
  and it is easy to see that $\varphi(U) = T$.
\end{proof}
As an example of the above bijection,
fix $\mu = 61$, $S = \{1,5\} \subseteq \{1,2,3,4,5 \}$, and
consider the pair
\begin{equation}\label{eq:tableauxposet2}
  P =
\begin{tikzpicture}[scale=.7,baseline=5]
\draw[fill] (0,2) circle (1mm); \node at (-.5,2) {$7$};
\draw[fill] (.8,1.5) circle (1mm); \node at (1.3,1.5) {$6$};
\draw[fill] (0,1) circle (1mm); \node at (-.5,1) {$5$};
\draw[fill] (0,0) circle (1mm); \node at (-.5,0) {$3$};
\draw[fill] (0,-1) circle (1mm); \node at (-.5,-1) {$1$};
\draw[fill] (.8,.5) circle (1mm); \node at (1.3,.5) {$4$};
\draw[fill] (.8,-.5) circle (1mm); \node at (1.3,-.5) {$2$};
\draw[-,thick] (0,1) -- (0,2);
\draw[-,thick] (.8,.5) -- (0,2);
\draw[-,thick] (.8,.5) -- (.8,1.5);
\draw[-,thick] (0,0) -- (.8,1.5);
\draw[-,thick] (0,1) -- (0,0);
\draw[-,thick] (0,0) -- (0,-1);
\draw[-,thick] (0,1) -- (.8,-.5);
\draw[-,thick] (0,-1) -- (.8,.5);
\draw[-,thick] (.8,.5) -- (.8,-.5);
\end{tikzpicture},
\qquad 
U = {\tableau[scY]{7|5,3,1,6,4,2}}\,,
\end{equation}
where $U$ is a column-strict $P$-tableau of shape $\mu$ which satisfies
$U_{1,i} >_P U_{1,i+1}$, for $i = 1,5$.
Corresponding to $S$
is the special ribbon diagram
\begin{equation}\label{eq:nonno}
Q = \raisebox{-4.5mm}{\includegraphics[height=11mm]{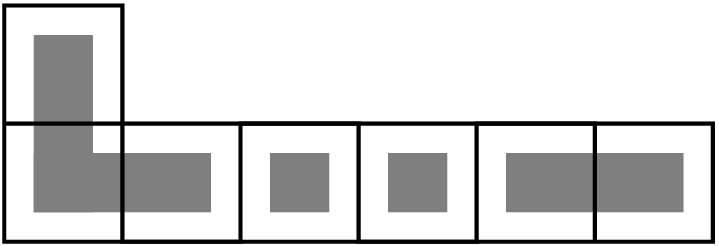}}
\end{equation}
of shape $\mu$ and type $3211$.
Superimposing $Q$ onto $U$
and transforming ribbons into columns,
we obtain the column-strict $P$-tableau
\begin{equation}\label{eq:nonno}
  \varphi(U) = {\tableau[scY]{7|5,4|3,2,1,6}}\,
\end{equation}
of shape $3211^\tr = 421$.

The above bijections allow for a new proof of Kaliszewski's
result~\cite[Prop.\ 4.3]{KaliHook}.
\begin{prop}\label{p:Kali}
  For any $n$-element poset $P$ and hook shape
  $k1^{n-k} \vdash n$,
  the evaluation
  $\chi^{\smash{k1^{n-k}}}(\inc(P))$
equals the number of standard $P$-tableaux of shape $k1^{n-k}$.
\end{prop}
\begin{proof}
  Fix $\mu = k1^{n-k}$,
  let $a(\mu)$ be the number of standard $P$-tableaux of shape $\mu$,
  and for each subset $S \subseteq [k-1]$, let $b(\mu,S)$ be the number of
  column-strict $P$-tableaux $U$ of shape $\mu$ which satisfy
  \begin{equation}\label{eq:Sdec}
    U_{1,j} >_P U_{1,j+1} \qquad \text{for all $j \in S$}.
    \end{equation}
  By the principle of
  inclusion-exclusion, these are related by
  \begin{equation}\label{eq:inclexcl}
    a(\mu) = \sum_{S \subseteq [k-1]} (-1)^{|S|} b(\mu,S).
  \end{equation}
  Each subset $S \subseteq [k-1]$ corresponds to a special ribbon diagram
  $Q = Q(\mu,S)$ of shape $\mu$ as described before Lemma~\ref{l:ribbon}.
  The partition $\lambda(\mu, S)$
  satisfies
  $|S| = k - \ell(\lambda(\mu,S))$.

  By Lemma~\ref{l:ribbon}, $b(\mu,S)$ is
  also the number of pairs $(Q,T)$ 
  with $Q$ a special ribbon diagram of shape $\mu$ and type $\lambda(\mu,S)$,
  and $T$ a column-strict $P$-tableau of shape $\lambda^\tr$.
  Thus we may rewrite (\ref{eq:inclexcl}) by summing over pairs $(S,U)$ and
  $(Q,T)$ satisfying the above conditions,
  \begin{equation*}
    a(\mu) = \sum_{(S,U)} (-1)^{|S|} = \sum_{(Q,T)} (-1)^{k - \ell(\type(Q))}.
  \end{equation*}
  Now collect terms in the last sum which correspond to
  special ribbon diagrams sharing the same partition $\type(Q) \vdash n$.
  Summing first over $\lambda \vdash n$
  and then over special ribbon diagrams $Q$ of shape $\mu$
  and type $\lambda$,
  we have that $a(\mu)$ is
  \begin{equation*}
    \sum_{\lambda \vdash n} \sum_{Q} (-1)^{k - \ell(\lambda)}
    (\# \text{column-strict $P$-tableaux of shape $\lambda^\tr$})
    = \sum_{\lambda \vdash n} \sum_Q \sgn(Q) \epsilon^\lambda(\inc(P)).
  \end{equation*}
By (\ref{eq:specribdi}) and (\ref{eq:epsilontochi}), this is
\begin{equation*}
  \sum_{\lambda \vdash n} K^{-1}_{\lambda,\mu^\tr\;} \epsilon^\lambda(\inc(P))
  = \chi^\mu(\inc(P)).
\end{equation*}

  \end{proof}

\ssec{Induced trivial characters / monomial coefficients of $\omega X_{\inc(P)}$}\label{ss:indtriv}
Given graph $G = (V,E)$ call a directed graph $O = (V,E')$ an {\em orientation}
of $G$ if $O$ is obtained from $G$ by replacing each
undirected edge $\{i,j\} \in E$
with exactly one of the directed edges $(i,j)$ or $(j,i)$.
Call $O$ {\em acyclic} if it has no directed cycles.
For example, a graph and two of its orientations
are 
\begin{equation*}
  G = \ntnsp
\begin{tikzpicture}[scale=.7,baseline=-5]
\draw[fill] (-1.2,.5) circle (1mm); \node at (-1.2,1) {$1$};
\draw[fill] (-.4,.5) circle (1mm); \node at (-.4,1) {$2$};
\draw[fill] (.4,.5) circle (1mm); \node at (.4,1) {$4$};
\draw[fill] (0,-.5) circle (1mm); \node at (0,-1) {$3$};
\draw[fill] (1.2,.5) circle (1mm); \node at (1.2,1) {$5$};
\draw[-,thick] (-1.2,.5) -- (0,-.5);
\draw[-,thick] (-.4,.5) -- (0,-.5);
\draw[-,thick] (.4,.5) -- (0,-.5);
\draw[-,thick] (1.2,.5) -- (0,-.5);
\draw[-,thick] (-1.2,.5) -- (-.4,.5);
\draw[-,thick] (-.4,.5) -- (.4,.5);
\draw[-,thick] (.4,.5) -- (1.2,.5);
\end{tikzpicture}\,,
\qquad
O_1 = \ntnsp
\begin{tikzpicture}[scale=.7,baseline=-5]
\draw[fill] (-1.2,.5) circle (1mm); \node at (-1.2,1) {$1$};
\draw[fill] (-.4,.5) circle (1mm); \node at (-.4,1) {$2$};
\draw[fill] (.4,.5) circle (1mm); \node at (.4,1) {$4$};
\draw[fill] (0,-.5) circle (1mm); \node at (0,-1) {$3$};
\draw[fill] (1.2,.5) circle (1mm); \node at (1.2,1) {$5$};
\draw[->,thick] (-1.2,.5) -- (-.15,-.43);
\draw[<-,thick] (-.35,.35) -- (0,-.5);
\draw[->,thick] (.4,.5) -- (0.05,-.35);
\draw[<-,thick] (1.09,.4) -- (0,-.5);
\draw[->,thick] (-1.2,.5) -- (-.55,.5);
\draw[->,thick] (-.4,.5) -- (.25,.5);
\draw[->,thick] (.4,.5) -- (1.05,.5);
\end{tikzpicture}\,,
\qquad
O_2 = \ntnsp
\begin{tikzpicture}[scale=.7,baseline=-5]
\draw[fill] (-1.2,.5) circle (1mm); \node at (-1.2,1) {$1$};
\draw[fill] (-.4,.5) circle (1mm); \node at (-.4,1) {$2$};
\draw[fill] (.4,.5) circle (1mm); \node at (.4,1) {$4$};
\draw[fill] (0,-.5) circle (1mm); \node at (0,-1) {$3$};
\draw[fill] (1.2,.5) circle (1mm); \node at (1.2,1) {$5$};
\draw[->,thick] (-1.2,.5) -- (-.15,-.43);
\draw[<-,thick] (-.35,.35) -- (0,-.5);
\draw[->,thick] (.4,.5) -- (0.05,-.35);
\draw[<-,thick] (1.09,.4) -- (0,-.5);
\draw[->,thick] (-1.2,.5) -- (-.55,.5);
\draw[<-,thick] (-.25,.5) -- (.4,.5);
\draw[->,thick] (.4,.5) -- (1.05,.5);
\end{tikzpicture}\,.
\end{equation*}
$O_2$ is acyclic; $O_1$ is not.  

By \cite{StanSymm}
we have for all graphs $G$ that
\begin{equation}\label{eq:etanPao}
  \eta^n(G) = \# \text{ acyclic orientations of $G$},
\end{equation}
and as a consequence (or by Proposition~\ref{p:Kali}) we have
for all posets $P$ that
\begin{equation}\label{eq:etanPdfree}
\eta^n(\inc(P))
    = \# \text{$P$-descent-free $P$-permutations}.
\end{equation}
When $P$ is a unit interval order labeled as in (\ref{eq:altrespect})
we also have~\cite[Thm.\,4.7]{CHSSkanEKL}
\begin{equation}\label{eq:etalambdaP}
  \eta^\lambda(\inc(P)) =
  \# \text{$P$-descent-free $P$-tableaux of shape $\lambda$}.
\end{equation}
We will extend this result to all
posets in Theorem~\ref{t:etalambdaP} and will
include combinatorial interpretations related to
$P$-excedance-free $P$-tableaux and acyclic orientations.
To do so, we consider some straightforward extensions of permutation statistics
to $P$-permutations.

Let $w$ be a $P$-permutation, and let
$\pexc(w)$ and $\paexc(w)$ be the numbers of $P$-excedances and
$P$-antiexcedances in $w$,
\begin{equation*}
  \pexc(w) = \# \{ i \,|\, w_i >_P i \},
  \qquad
  \paexc(w) = \# \{ i \,|\, w_i <_P i \}.
\end{equation*}
Let $\pdes(w)$ and $\pasc(w)$ be the numbers of $P$-descents and $P$-ascents
in $w$,
\begin{equation*}
  \pdes(w) = \# \{ i \,|\, w_i >_P w_{i+1} \},
  \qquad
  \pasc(w) = \# \{ i \,|\, w_i <_P w_{i+1} \}.
\end{equation*}

Define the {\em standard cycle notation} of $w \in \sn$ to be
the cycle notation in which cycles are listed in increasing order of their
greatest elements, and these greatest elements are listed first in each cycle.
Let $\sigma: \sn \rightarrow \sn$ be the bijection~\cite[\S 1.3]{StanEC1}
defined by setting
$\sigma(w)$ equal to the permutation whose one-line notation is obtained
by erasing parentheses from the standard cycle notation of $w$.
For example, to compute $\sigma(5243761)$, we write $5243761$
in standard cycle notation as $(2)(4,3)(6)(7,1,5)$,
since the greatest elements of the cycles satisfy $2 < 4 < 6 < 7$.
Then we erase parentheses to obtain $2436715$.

The following result is a strengthening of \cite[Exercise\,3.60c]{StanEC1}.
It first appeared with a different proof in
\cite[Thm.\,4.6]{EinarThesis}.
\begin{prop}\label{p:equidist}
  For any
  poset $P$, the
  statistics $\pdes$, $\pasc$, $\pexc$, $\paexc$
  are equally distributed on the set of all $P$-permutations.
\end{prop}
\begin{proof}
  ($\pdes \sim \pasc$) We have $\pasc(w) = \pdes(w_n \cdots w_1)$.

  \noindent ($\pexc \sim \paexc$) We have $\paexc(w) = \pexc(w^{-1})$.

  \noindent ($\pdes \sim \paexc$)
  Assume first that $P$ is naturally labeled (\ref{eq:natlabel}).
  We claim that the map $\sigma$ satisfies $\pdes(\sigma(w)) = \paexc(w)$.
  To see this, write $w$ in standard cycle notation
  and $\sigma(w)$ in one-line notation as
  \begin{equation*}
    \begin{aligned}
      w &= (a_1, a_2, \dotsc, a_{i_1})
      (a_{i_1 +1}, a_{i_1 + 2}, \dotsc, a_{i_2}) \cdots
      (a_{i_{k-1} +1}, a_{i_{k-1} + 2}, \dotsc, a_{i_k} = a_n),\\
      \sigma(w) &= a_1, a_2, \dotsc, a_{i_1},
      a_{i_1 +1}, a_{i_1 + 2}, \dotsc, a_{i_2}, \dotsc,
      a_{i_{k-1} +1}, a_{i_{k-1} + 2}, \dotsc, a_{i_k} = a_n.
    \end{aligned}
  \end{equation*}
  Suppose that $j$ is a $P$-descent of $\sigma(w)$.
  By the natural labeling of $P$, we have $a_j > a_{j+1}$ and therefore
  $a_j$ can not appear last in its cycle
  in the standard cycle notation for $w$.
  Thus we have $a_j >_P a_{j+1} = w(a_j)$ and
  position $a_j$ is a $P$-antiexcedance of $w$.
  Thus position $a_j$ is a $P$-antiexcedance of $w$.
  Now suppose that $j$ is not a $P$-descent of $\sigma(w)$.
   Then in the standard cycle notation for $w$
  we have either that
  $a_j, a_{j+1}$ appear consecutively in a cycle and satisfy
  $a_j \not >_P a_{j+1} = w(a_j)$,
  or that $a_j$ appears last in its cycle
  and satisfies $a_j \leq w(a_j)$.
  By the natural labeling of $P$, this last inequality implies
  $a_j \not >_P w(a_j)$. Thus in both cases 
  position $a_j$ is not a $P$-antiexcedance of $w$.

  Now assume that $P$ is nonnaturally labeled,
  and let $P'$ be a naturally labeled copy of $P$.
  Then for some $u \in \sn$, the poset isomorphism $P \rightarrow P'$
  is given by $i \mapsto u_i$,
  and the bijection $w \mapsto uw$ from $\sn$ to itself
  satisfies $\aexc_P(w) = \aexc_{P'}(uw)$.
\end{proof}

Given a graph $G$ on $n$ vertices
and an ordered set partition $(I_1, \dotsc, I_r)$ of $[n]$
of type $\lambda \vdash n$,
call the sequence $(G_{I_1}, \dotsc, G_{I_r})$
an {\em ordered induced subgraph partition} of $G$ of type $\lambda$.
Let $\oisp_\lambda(G)$ be the set of such sequences, and
define an {\em acyclic orientation} of an element of $\oisp_\lambda(G)$ to be
a sequence $(O_1, \dotsc, O_r)$,
where $O_j$ is an acyclic orientation of $G_{I_j}$.
For example, consider the ordered set partition
$(234,15)$
of type $(3,2)$.
A graph $G$,
its ordered induced subgraph partition
$(G_{234}, G_{15})$, and one acyclic orientation $(O_1, O_2)$ of this are
\begin{equation}\label{eq:oispao}
  G = \ntnsp
\begin{tikzpicture}[scale=.7,baseline=-5]
\draw[fill] (-1.2,.5) circle (1mm); \node at (-1.2,1) {$1$};
\draw[fill] (-.4,.5) circle (1mm); \node at (-.4,1) {$2$};
\draw[fill] (.4,.5) circle (1mm); \node at (.4,1) {$4$};
\draw[fill] (0,-.5) circle (1mm); \node at (0,-1) {$3$};
\draw[fill] (1.2,.5) circle (1mm); \node at (1.2,1) {$5$};
\draw[-,thick] (-1.2,.5) -- (0,-.5);
\draw[-,thick] (-.4,.5) -- (0,-.5);
\draw[-,thick] (.4,.5) -- (0,-.5);
\draw[-,thick] (1.2,.5) -- (0,-.5);
\draw[-,thick] (-1.2,.5) -- (-.4,.5);
\draw[-,thick] (-.4,.5) -- (.4,.5);
\draw[-,thick] (.4,.5) -- (1.2,.5);
\end{tikzpicture}\ntksp,
\qquad
  (G_{234}, G_{15}) = \ntnsp
\Bigg(
\begin{tikzpicture}[scale=.7,baseline=-5]
\draw[fill] (-.4,.5) circle (1mm); \node at (-.4,1) {$2$};
\draw[fill] (.4,.5) circle (1mm); \node at (.4,1) {$4$};
\draw[fill] (0,-.5) circle (1mm); \node at (0,-1) {$3$};
\draw[-,thick] (-.4,.5) -- (0,-.5);
\draw[-,thick] (.4,.5) -- (0,-.5);
\draw[-,thick] (-.4,.5) -- (.4,.5);
\end{tikzpicture}\nTksp,\
\begin{tikzpicture}[scale=.7,baseline=-5]
\draw[fill] (0,.5) circle (1mm); \node at (0,1) {$5$};
\draw[fill] (0,-.5) circle (1mm); \node at (0,-1) {$1$};
\end{tikzpicture}
\Bigg),
\qquad
  (O_1, O_2) = \ntnsp
\Bigg(
\begin{tikzpicture}[scale=.7,baseline=-5]
\draw[fill] (-.4,.5) circle (1mm); \node at (-.4,1) {$2$};
\draw[fill] (.4,.5) circle (1mm); \node at (.4,1) {$4$};
\draw[fill] (0,-.5) circle (1mm); \node at (0,-1) {$3$};
\draw[<-,thick] (-.35,.35) -- (0,-.5);
\draw[->,thick] (.4,.5) -- (0.05,-.35);
\draw[<-,thick] (-.25,.5) -- (.4,.5);
\end{tikzpicture}\nTksp,\
\begin{tikzpicture}[scale=.7,baseline=-5]
\draw[fill] (0,.5) circle (1mm); \node at (0,1) {$5$};
\draw[fill] (0,-.5) circle (1mm); \node at (0,-1) {$1$};
\end{tikzpicture}
\Bigg).
\end{equation}

\begin{thm}\label{t:etalambdaP}
  For any
  poset $P$ and partition $\lambda = (\lambda_1,\dotsc,\lambda_r) \vdash |P|$,
  the number $\eta^\lambda(\inc(P))$ has the 
  combinatorial interpretations
  \begin{enumerate}
  \item $\#$ $P$-descent-free $P$-tableaux of shape $\lambda$,
  \item $\#$ $P$-excedance-free $P$-tableaux of shape $\lambda$,
  \item $\#$ acyclic orientations of sequences $(\inc(P_{I_1}), \dotsc, \inc(P_{I_r})) \in \oisp_\lambda(\inc(P))$.
  \end{enumerate}
\end{thm}
\begin{proof}
  Let $\lambda = (\lambda_1,\dotsc,\lambda_r)$.
  Since $h_\lambda = h_{\lambda_1} \cdots h_{\lambda_r}$, Lemma~\ref{l:efactor}
  implies that
  induced trivial characters satisfy
  \begin{equation}\label{eq:etaid}
    \eta^\lambda(\inc(P)) = \ntksp \sum_{(I_1,\dotsc,I_r)} \ntksp
    \eta^{\lambda_1}(\inc(P_{I_1})) \cdots \eta^{\lambda_r}(\inc(P_{I_r})),
  \end{equation}
  where the sum is over all ordered set partitions of $|P|$ of type $\lambda$.
  By (\ref{eq:etanPdfree}), we have for all $n$-element posets $P$ that
  $\eta^n(\inc(P))$ is the number of $P$-descent-free $P$-permutations.
  This implies interpretation (1) of the theorem, and then
  Proposition~\ref{p:equidist} implies interpretation (2).

  Interpretation (3) follows from the known bijection between
  acyclic orientations of $\inc(P)$
  and $P$-descent-free $P$-permutations.  (See, e.g., \cite[\S 4]{AthanPSE},
  \cite[Exercise\,3.60b]{StanEC1}.)
  Specifically, given an acyclic orientation
  $(O_1,\dotsc,O_r)$ of a sequence
  $(\inc(P_{I_1}), \dotsc, \inc(P_{I_r}))$ in $\oisp_\lambda(\inc(P))$, 
  create row-semistrict $P$-tableau $U$ of shape $\lambda$ as follows.

  For $i = 1,\dotsc,r$, do
    \begin{enumerate}
      \item Initialize $U_i$ to be the empty Young diagram of shape $\lambda_i$.
      \item While $O_i$ is not empty do
        \begin{enumerate}
        \item Let $j$ be the minimum element of the chain in $P$ consisting
          of the source vertices in $O_i$.
        \item Update $O_i$ by removing $j$ and its outgoing edges.
        \item Update $U_i$ by placing $j$ in its leftmost empty box.
        \end{enumerate}
    \end{enumerate}

  Given a row-semistrict $P$-tableau $U$ of shape $\lambda$,
  create an induced subgraph partition of $\inc(P)$ and an
  acyclic orientation $(O_1,\dotsc,O_r)$ of it as follows.

   For $i=1,\dotsc,r$, do
    \begin{enumerate}
      \item Let the vertices of $O_i$ be the set of elements in $U_i$.
      \item For all pairs $(j,k)$ of entries of $U_i$ which are incomparable
        in $P$, if $j$ precedes $k$ in $U_i$, then create a directed edge
        from $j$ to $k$ in $O_i$.
    \end{enumerate}
\end{proof}

The special case of Theorem~\ref{t:etalambdaP} (2)
corresponding to $P$ a unit interval order and $\lambda = n$
has an interpretation in terms of the Bruhat order on $\sn$.
\begin{cor}\label{c:bruhatexcfree}
  Let $P$ be a unit interval order on $[n]$ labeled as in (\ref{eq:altrespect})
  and let $w \in \sn$ be the corresponding $312$-avoiding permutation
  as in (\ref{eq:uioto312avoid}).
  Then we have
  \begin{equation*}
    \{ v \in \sn \,|\, \exc_P(v) = 0 \} = \{ v \in \sn \,|\, v \leq w \}.
  \end{equation*}
\end{cor}
\begin{proof}
  Define the matrix
  $A = (a_{i,j})$ by
  \begin{equation*}
    a_{i,j} = \begin{cases}
      0 &\text{if $i <_P j$}\\
      1 &\text{otherwise}.
    \end{cases}
    \end{equation*}
  By \cite[Lem.\,5.3 (3)]{SkanNNDCB}, the product $\permmon av$ is
  $1$ if $v \leq w$ and is $0$ otherwise.  But $\permmon av = 1$ if and only
  if $i \not <_P v_i$ for $i = 1,\dotsc,n$, i.e., if and only if
  $v$ is $P$-excedance free.
  \end{proof}

\ssec{Power sum traces / scaled power sum coefficients of $\omega X_{\inc(P)}$}
It is known that we have
  $\psi^\lambda(\inc(P)) \geq 0$
for all $P$~\cite{StanSymm}, and
\begin{equation}\label{eq:psilambdaP}
  \begin{aligned}
    \psi^\lambda(\inc(P))
    &= \# \text{ cyclically row-semistrict $P$-tableaux of shape $\lambda$}\\
    &= \# \text{ $P$-record-free, row-semistrict $P$-tableaux of shape $\lambda$}
    \end{aligned}
\end{equation}
for all unit interval orders $P$ labeled as in
(\ref{eq:altrespect})~\cite[Thm.\,4]{AthanPSE},
\cite[Thm.\,4.7]{CHSSkanEKL},
\cite[\S 7]{SWachsChromQF}.
We will extend these results to all posets 
in Theorem~\ref{t:psilambdaP}, and will include more combinatorial
interpretations involving $\inc(P)$ and a related directed graph.
Define $\ngr(P)$ to be the directed graph whose vertices
are the elements of $P$ and whose edges are the
ordered pairs $\{ (i,j) \in P^2 \,|\, i \not >_P j \}$,
including loops $(i,i)$ for all $i \in P$.
For example,
a poset $P$
and related directed graph
$\ngr(P)$ are
\begin{equation}\label{eq:ngr}
  P = \ntnsp
\begin{tikzpicture}[scale=.7,baseline=-5]
\draw[fill] (0,.5) circle (1mm); \node at (0,1) {$4$};
\draw[fill] (0,-.5) circle (1mm); \node at (0,-1) {$1$};
\draw[fill] (.8,.5) circle (1mm); \node at (.8,1) {$5$};
\draw[fill] (.8,-.5) circle (1mm); \node at (.8,-1) {$2$};
\draw[fill] (1.6,0) circle (1mm); \node at (1.6,-.5) {$3$};
\draw[-,thick] (0,.5) -- (0,-.5);
\draw[-,thick] (.8,.5) -- (0,-.5);
\draw[-,thick] (.8,.5) -- (.8,-.5);
\end{tikzpicture}\,,
\qquad 
  \ngr(P) = 
\begin{tikzpicture}[scale=.7,baseline=-5]
\draw[fill] (0,0) circle (1mm); \node at (0,-.85) {$3$};
\draw[fill] (-1,-1) circle (1mm); \node at (-1.6,-1.6) {$1$};
\draw[fill] (-1,1) circle (1mm); \node at (-1.6,.8) {$2$};
\draw[fill] (1,1) circle (1mm); \node at (1.6,.8) {$4$};
\draw[fill] (1,-1) circle (1mm); \node at (1.6,-1.6) {$5$};
\draw[->,thick] (-1,1) -- (-1,-.85);
\draw[->,thick] (-1,-1) to [out=135,in=-120] (-1.05,.85);
\draw[<-,thick] (-.85,1) -- (1,1);
\draw[->,thick] (-1,1) to [out=45,in=150] (.85,1.05);
\draw[<-,thick] (1,.85) -- (1,-1);
\draw[->,thick] (1,1) to [out=-45,in=60] (1.05,-.85);
\draw[->,thick] (-1,-1) to [out=-45,in=-150] (.85,-1.05);
\draw[-,thick] (-1,-1) to [out=160,in=-135] (-1.6,1.6);
\draw[->,thick] (-1.6,1.6) to [out=45,in=110] (.9,1.1);
\draw[-,thick] (-1,1) to [out=70,in=135] (1.6,1.6);
\draw[->,thick] (1.6,1.6) to [out=-45,in=20] (1.1,-.9);
\draw[<-,thick] (-.85,-1) to [out=20,in=-145](0,0);
\draw[->,thick] (-1,-1) to [out=65,in=-165](-.15,-.04);
\draw[->,thick] (-1,1) to [out=-25,in=115] (-.05,.15);
\draw[<-,thick] (-.9,.9) to [out=-65,in=165] (0,0);
\draw[->,thick] (1,1) to [out=-105,in=25] (.1,.05);
\draw[<-,thick] (.9,.9) to [out=-155,in=75] (0,0);
\draw[<-,thick] (.85,-1) to [out=155,in=-35] (0,0);
\draw[->,thick] (1,-1) to [out=105,in=-15] (.15,-.04);
\draw[-,thick] (-1,1) to [out=90,in=0] (-1.2,1.4);
\draw[-,thick] (-1.2,1.4) to [out=180,in=90] (-1.4,1.2);
\draw[->,thick] (-1.4,1.2) to [out=-90,in=150] (-1.15,1);
\draw[-,thick] (-1,-1) to [out=-90,in=0] (-1.2,-1.4);
\draw[-,thick] (-1.2,-1.4) to [out=180,in=-90] (-1.4,-1.2);
\draw[->,thick] (-1.4,-1.2) to [out=90,in=-150] (-1.15,-1);
\draw[-,thick] (1,1) to [out=90,in=180] (1.2,1.4);
\draw[-,thick] (1.2,1.4) to [out=0,in=90] (1.4,1.2);
\draw[->,thick] (1.4,1.2) to [out=-90,in=30] (1.15,1);
\draw[-,thick] (1,-1) to [out=-90,in=180] (1.2,-1.4);
\draw[-,thick] (1.2,-1.4) to [out=0,in=-90] (1.4,-1.2);
\draw[->,thick] (1.4,-1.2) to [out=90,in=-30] (1.15,-1);
\draw[-,thick] (0,0) to [out=-60,in=60] (.15,-.5);
\draw[-,thick] (.15,-.5) to [out=-120,in=-60] (-.1,-.5);
\draw[->,thick] (-.1,-.5) to [out=120,in=-110] (0,-.1);
\end{tikzpicture}\,.
\end{equation}

Given a directed graph $D$ on $n$ vertices
and a partition $\lambda = (\lambda_1,\dotsc,\lambda_r) \vdash n$,
call a sequence $(H_1,\dotsc,H_r)$ of vertex-disjoint
subdigraphs of $D$
a {\em disjoint cyclic vertex cover of $D$ of type $\lambda$} if $H_j$ 
is isomorphic to the cycle graph $C_{\lambda_j}$ for all $j$.
For example, three disjoint cyclic vertex covers of the graph $\ngr(P)$
in (\ref{eq:ngr}) are
\begin{equation*}
\Bigg( \begin{tikzpicture}[scale=.7,baseline=-5]
\draw[fill] (.8,.5) circle (1mm); \node at (.8,1) {$2$};
\draw[fill] (.8,-.5) circle (1mm); \node at (.8,-1) {$1$};
\draw[fill] (1.6,0) circle (1mm); \node at (1.6,-.5) {$3$};
\draw[<-,thick] (.8,-.35) -- (.8,.5);
\draw[->,thick] (.8,-.5) -- (1.48,-.07);
\draw[<-,thick] (.93,.42) -- (1.6,0);
\end{tikzpicture}\ntnsp,
\begin{tikzpicture}[scale=.7,baseline=-5]
\draw[fill] (.8,.5) circle (1mm); \node at (.8,1) {$4$};
\draw[fill] (.8,-.5) circle (1mm); \node at (.8,-1) {$5$};
\draw[->,thick] (.8,-.5) to [out=120,in=-105] (.75,.4);
\draw[->,thick] (.8,.5) to [out=-60,in=75] (.85,-.4);
\end{tikzpicture}\, \Bigg),
\qquad
\Bigg( \begin{tikzpicture}[scale=.7,baseline=-5]
\draw[fill] (.8,.5) circle (1mm); \node at (.8,1) {$2$};
\draw[fill] (.8,-.5) circle (1mm); \node at (.8,-1) {$1$};
\draw[fill] (1.6,0) circle (1mm); \node at (1.6,-.5) {$3$};
\draw[->,thick] (.8,-.5) -- (.8,.35);
\draw[<-,thick] (.92,-.43) -- (1.6,0);
\draw[->,thick] (.8,.5) -- (1.48,.07);
\end{tikzpicture}\ntnsp,
\begin{tikzpicture}[scale=.7,baseline=-5]
\draw[fill] (0,0) circle (1mm); \node at (0,.5) {$4$};
\draw[-,thick] (0,0) to [out=-60,in=60] (.15,-.5);
\draw[-,thick] (.15,-.5) to [out=-120,in=-60] (-.1,-.5);
\draw[->,thick] (-.1,-.5) to [out=120,in=-110] (0,-.1);
\end{tikzpicture}\ntnsp,
\begin{tikzpicture}[scale=.7,baseline=-5]
\draw[fill] (0,0) circle (1mm); \node at (0,.5) {$5$};
\draw[-,thick] (0,0) to [out=-60,in=60] (.15,-.5);
\draw[-,thick] (.15,-.5) to [out=-120,in=-60] (-.1,-.5);
\draw[->,thick] (-.1,-.5) to [out=120,in=-110] (0,-.1);
\end{tikzpicture}\, \Bigg),
\qquad
\Bigg( \begin{tikzpicture}[scale=.7,baseline=-5]
\draw[fill] (.8,.5) circle (1mm); \node at (.8,1) {$2$};
\draw[fill] (.8,-.5) circle (1mm); \node at (.8,-1) {$1$};
\draw[fill] (1.6,0) circle (1mm); \node at (1.6,-.5) {$3$};
\draw[->,thick] (.8,-.5) -- (.8,.35);
\draw[<-,thick] (.92,-.43) -- (1.6,0);
\draw[->,thick] (.8,.5) -- (1.48,.07);
\end{tikzpicture}\ntnsp,
\begin{tikzpicture}[scale=.7,baseline=-5]
\draw[fill] (0,0) circle (1mm); \node at (0,.5) {$5$};
\draw[-,thick] (0,0) to [out=-60,in=60] (.15,-.5);
\draw[-,thick] (.15,-.5) to [out=-120,in=-60] (-.1,-.5);
\draw[->,thick] (-.1,-.5) to [out=120,in=-110] (0,-.1);
\end{tikzpicture}\ntnsp,
\begin{tikzpicture}[scale=.7,baseline=-5]
\draw[fill] (0,0) circle (1mm); \node at (0,.5) {$4$};
\draw[-,thick] (0,0) to [out=-60,in=60] (.15,-.5);
\draw[-,thick] (.15,-.5) to [out=-120,in=-60] (-.1,-.5);
\draw[->,thick] (-.1,-.5) to [out=120,in=-110] (0,-.1);
\end{tikzpicture}\, \Bigg).
\end{equation*}
These have types $32$ and $311$.
Observe that a collection of vertices does not completely specify a cycle,
as this example shows two different cycles on the vertices $\{1, 2, 3\}$.

To prove our results, we will use the transition matrix which
relates the elementary and power sum bases of $\Lambda_n$.
In particular,
\begin{equation}\label{eq:ptoe}
  p_n = \sum_{\mu \vdash n} (-1)^{n-\ell(\mu)} c_{\mu} e_\mu,  
\end{equation}
where $c_\mu$ equals the number of subgraphs of
the (labeled) cycle graph
\begin{equation*}
C_n = ([n],E), \qquad  E = \{(i,i+1) \,|\, 1 \leq i \leq n-1\} \cup \{ (n,1) \},
\end{equation*}
whose connected components are
paths on $\mu_1,\dotsc,\mu_k$ vertices.
Clearly such subgraphs correspond bijectively
to subsets $S \subseteq [n]$,
\begin{equation*}
  S \quad \longleftrightarrow \quad
  C_{n,S} = ([n], E_S) \text{ with } E_S = \{ (i,j) \in E \,|\, i \in S \},
\end{equation*}
and we define
$\mu(S)$ to be the weakly decreasing sequence of component cardinalities of
$C_{n,S}$.
We will also use the set of $P$-permutations
whose (cyclic) $P$-descent set contains $S$,
\begin{equation}\label{eq:bs}
  \mathcal B(S) = \{ w \in \sn \,|\, w_i >_P w_{i+1}
  \text{ (or } w_i = w_n >_P w_1 ) \text{ for all } i \in S \}.
  \end{equation}
\begin{lem}\label{l:bs}
  For any $n$-element poset $P$ and subset $S \subseteq [n]$,
  the permutations in $\mathcal B(S)$ correspond bijectively
  to column-strict $P$-tableaux of shape $\mu(S)$.  In particular, we have
    $|\mathcal B(S)| = \epsilon^{\mu(S)}(\inc(P))$.
\end{lem}
\begin{proof}
   Fix $w \in \mathcal B(S)$.  For
  each maximal interval
  $[i,j]$
  (mod $n$)
  with $i,\dotsc,j \in S$,
  we have the chain
  $w_{i} >_P \cdots >_P w_{j+1}$;
  for $i-1, i \not \in S$, we have the one-element chain $w_i$.
  Let $\mu = \mu(S)$ and insert the
    $\ell(\mu)
    = n - |S|$
  chains
  into the columns of a Young diagram of shape $\mu^\tr$
  to obtain a column-strict $P$-tableau. 
  For chains of equal cardinalities, fill the leftmost available
  column of the tableau with the leftmost available chain in $w$
  (considering $\cdots >_P w_n >_P w_1 >_P \cdots$
  to be the leftmost chain of all, if it exists).
  It is easy to see that this map is invertible.
\end{proof}
As an example of the above bijection,
consider the poset, subset, and $P$-permutation
\begin{equation}\label{eq:tableauxposet3}
  P =
\begin{tikzpicture}[scale=.7,baseline=5]
\draw[fill] (0,2) circle (1mm); \node at (-.5,2) {$7$};
\draw[fill] (.8,1.5) circle (1mm); \node at (1.3,1.5) {$6$};
\draw[fill] (0,1) circle (1mm); \node at (-.5,1) {$5$};
\draw[fill] (0,0) circle (1mm); \node at (-.5,0) {$3$};
\draw[fill] (0,-1) circle (1mm); \node at (-.5,-1) {$1$};
\draw[fill] (.8,.5) circle (1mm); \node at (1.3,.5) {$4$};
\draw[fill] (.8,-.5) circle (1mm); \node at (1.3,-.5) {$2$};
\draw[-,thick] (0,1) -- (0,2);
\draw[-,thick] (.8,.5) -- (0,2);
\draw[-,thick] (.8,.5) -- (.8,1.5);
\draw[-,thick] (0,0) -- (.8,1.5);
\draw[-,thick] (0,1) -- (0,0);
\draw[-,thick] (0,0) -- (0,-1);
\draw[-,thick] (0,1) -- (.8,-.5);
\draw[-,thick] (0,-1) -- (.8,.5);
\draw[-,thick] (.8,.5) -- (.8,-.5);
\end{tikzpicture},
\qquad 
S = \{1, 4, 5, 7 \},
\qquad
w = 5316427,
\end{equation}
with $w \in \mathcal B(\{1,4,5,7\})$ because
the cyclic $P$-descents of $w$ include $1, 4, 5, 7$:
\begin{equation*}
  w_1 >_P w_2, \quad w_4 >_P w_5, \quad w_5 >_P w_6, \quad w_7 >_P w_1.
\end{equation*}
Combining these 
$P$-descents to form chains (and collecting leftover $1$-chains), we have
\begin{equation*}
  w_7 >_P w_1 >_P w_2 = 753,\qquad
  w_3 = 1, \qquad
  w_4 >_P w_5 >_P w_6 = 642,\qquad 
\end{equation*}
which we can insert in order of weakly decreasing cardinality
into a column-strict $P$-tableau of shape $331^\tr = 322$
\begin{equation*}
  {\tableau[scY]{7,6|5,4|3,2,1}}\,,
\end{equation*}
where we have broken the tie between $3$-element chains
by inserting the leftmost chain (the one containing $w_1$) first.

Now we may interpret $\psi^\lambda(\inc(P))$ as follows.
\begin{thm}\label{t:psilambdaP}
  For any
  poset $P$ and partition
  $\lambda  = (\lambda_1,\dotsc,\lambda_r) \vdash |P|$,
  the number $\psi^\lambda(\inc(P))$ has the 
  combinatorial interpretations
  \begin{enumerate}
  \item $\#$ cyclically row-semistrict $P$-tableaux of shape $\lambda$,
  \item $\#$ $P$-record-free, row-semistrict $P$-tableaux of shape $\lambda$,
  \item $\#$ disjoint cyclic vertex covers of $\ngr(P)$ of type $\lambda$,
  \item $\#$ acyclic orientations $(O_1,\dotsc,O_r)$ of
    sequences $(\inc(P_{I_1}),\dotsc,\inc(P_{I_r})) \in \oisp_\lambda(\inc(P))$
    in which each oriented component $O_i$
    has exactly one source.
  \end{enumerate}
\end{thm}
\begin{proof}
  Since $p_\lambda = p_{\lambda_1} \ntnsp\cdots p_{\lambda_r}$, Lemma~\ref{l:efactor}
  implies that
  power sum traces satisfy
  \begin{equation}\label{eq:psiid}
    \psi^\lambda(\inc(P)) = \ntksp \sum_{(I_1,\dotsc,I_r)} \ntksp
    \psi^{\lambda_1}(\inc(P_{I_1})) \cdots \psi^{\lambda_r}(\inc(P_{I_r})),
  \end{equation}
  where the sum is over all ordered set partitions of $|P|$ of type $\lambda$.

  \noindent (1) We claim that for any $n$-element poset $P$, 
  $\psi^n(\inc(P))$ is the number of cyclically row-semistrict
  $P$-permutations ($P$-descent-free $w_1 \cdots w_n$ with
  $w_n \not >_P w_1$).
  To see this, let $a$ be the number of such $P$-permutations
  and define
  $\mathcal B(S)$ as in (\ref{eq:bs}).
  By the principle of inclusion/exclusion and Lemma~\ref{l:bs},
  the cardinalities $a$ and $|\mathcal B(S)|$ are related by
  \begin{equation*}
    a = \sum_{S \subseteq [n]} (-1)^{|S|} |\mathcal B(S)| 
    = \sum_{S \subseteq [n]} (-1)^{|n - \ell(\mu(S))|} \epsilon^{\mu(S)}(\inc(P)).
  \end{equation*}
  By (\ref{eq:ptoe})
  the number of distinct subsets
  $T \subseteq [n]$
  satisfying $\mu(T) = \mu(S)$
  is $c_{\mu(S)}$. Therefore we have
  \begin{equation*}
    a
    = \sum_{\mu \vdash n}(-1)^{n - \ell(\mu)}c_{\mu} \epsilon^\mu(\inc(P))
    = \psi^n(\inc(P)),
  \end{equation*}
  as desired.

  \noindent (2) Stanley~\cite[Thm.\,3.3]{StanSymm} showed that
  $\phi^n(P) = \psi^n(P)$ equals the number of acyclic orientations of
  $\inc(P)$ having exactly one sink.  Reversing all edges in such an
  orientation and applying the bijection at the end of
  the proof of Theorem~\ref{t:etalambdaP},
  we have that $\psi^n(P)$ equals the number
  of $P$-record-free $P$-descent-free $P$-permutations.
  Then by (\ref{eq:psiid}) we have the desired result.
  
  \noindent (3) Let $U$ be a cyclically row-semistrict $P$-tableau of
  shape $\lambda$.  Each pair of horizontally
  adjacent entries $(U_{i,j}, U_{i,j+1})$ (or $(U_{i,\lambda_i}, U_{i,1})$) in $U$
  corresponds to an edge in $\ngr(P)$, and the row $U_i$ corresponds to a
  cycle.  Removing all other edges from
  $\ngr(P)$ and listing the cycles in order of their corresponding rows of $U$,
  we obtain the desired disjoint cyclic vertex cover.

  \noindent (4) As described above, $\psi^n(P)$ equals the number of
  acyclic orientations of $\inc(P)$ having exactly one source.
  Now (\ref{eq:psiid}) gives the desired result.
\end{proof}



\ssec{Monomial traces / elementary coefficients of $X_{\inc(P)}$}
While $\phi^\lambda(\inc(P))$ is negative for some posets $P$,
Stembridge and Stanley conjectured~\cite[Conj.\,5.5]{StanStemIJT}
that for $(\mathbf3 + \mathbf1$)-free posets $P$
we have
\begin{equation}\label{eq:phipos}
  \phi^\lambda(\inc(P)) \geq 0.
\end{equation}
By \cite[Thm.\,5.1]{GPModRel}, this conjecture is equivalent to the
assertion that (\ref{eq:phipos}) holds when $P$ is a unit interval order.
Using this fact, we may state special cases of the conjecture which are known
to be true.
In particular, the following 
conditions on $\lambda$ and/or $P$ imply
(\ref{eq:phipos}).
\begin{enumerate}
\item 
  $P$ $\mathbf3$-free~\cite[Thm.\,5.3]{GPModRel}, \cite[Cor.\,3.6]{StanSymm}.
\item $P$ a $\mathbf4$-free unit interval order~\cite[Thm.\,1.8]{CHPos}.
\item 
  $\lambda$ a {\em rectangular partition}, i.e., 
  $\lambda_1 = \lambda_2 = \cdots = k$,
  and $P$ ($\mathbf3 + \mathbf 1$)-free
  \cite[Thm.\,2.8]{StemConj}.
\item $P$ a unit interval order with a ($\lambda_1 + 1$)-element
  antichain~\cite[Thm.\,3.7, Prop.\,10.1]{CHSSkanEKL}.
\item $P$ a unit interval order with component sizes of $\inc(P)$ not refining
  $\lambda$~\cite[Prop.\,10.2]{CHSSkanEKL}.
\item 
  $\lambda_1 \leq 2$ and $P$ ($\mathbf3 + \mathbf 1$)-free
  \cite[Thm.\,10.3]{CHSSkanEKL}.
\item 
  $P$ defined on $[n]$ by
  $i <_P j$ if $i + 1 < j$
  (as integers)
  \cite[p.\,242]{CarlitzSVRiseFallLevel},
  \cite[Prop.\,5.3]{StanSymm}.
\item 
  $P$ defined on $[n]$ by
  $i <_P j$ if $i + 2 < j$ (as integers) \cite{DahlbergNewFormula}.
\item
  $P$ defined on $[n]$ by
  $i <_P j$ if $i + n - 3 < j$ (as integers) \cite{SWachsChromQF}.
\end{enumerate}
Conditions (4) and (5) above more specifically imply that we have
$\phi^\lambda(\inc(P)) = 0$.
Another related result concerns sums of monomial
traces~\cite[Thm.\,3.3]{StanSymm}.
\begin{prop}\label{p:stanksources}
  For all graphs $G$ on $n$ vertices we have 
  \begin{equation*}
    \sumsb{\lambda\vdash n\\ \ell(\lambda) = k} \phi^\lambda(G)
    = \# \text{ acyclic orientations of $G$ having $k$ sources},
\end{equation*}
and for all $n$-element posets $P$ we have
  \begin{equation*}
    \sumsb{\lambda\vdash n\\ \ell(\lambda) = k} \phi^\lambda(\inc(P))
    = \# \text{ $P$-descent-free $P$-permutations $w$ with $k$ $P$-records}.
  \end{equation*}
\end{prop}
Since
$\phi^{1^n} = \epsilon^n$ and $\phi^n = \psi^n$, it is tempting
to conjecture a formula for $\phi^\lambda(\inc(P))$ which combines
column-strictness of Equation~(\ref{eq:epsilonlambdaP}) with
one of the conditions of Theorem~\ref{t:psilambdaP}.
Two obvious combinations do not in general give correct formulas.
For example, the following poset $P$ and
monomial trace evaluations
\begin{equation*}
  P =
\begin{tikzpicture}[scale=.7,baseline=-5]
\draw[fill] (0,1) circle (1mm); \node at (-.5,1) {$5$};
\draw[fill] (0,0) circle (1mm); \node at (-.5,0) {$3$};
\draw[fill] (0,-1) circle (1mm); \node at (-.5,-1) {$1$};
\draw[fill] (.8,.5) circle (1mm); \node at (1.3,.5) {$4$};
\draw[fill] (.8,-.5) circle (1mm); \node at (1.3,-.5) {$2$};
\draw[-,thick] (0,1) -- (0,0);
\draw[-,thick] (0,0) -- (0,-1);
\draw[-,thick] (0,1) -- (.8,-.5);
\draw[-,thick] (0,-1) -- (.8,.5);
\draw[-,thick] (.8,.5) -- (.8,-.5);
\end{tikzpicture},
\qquad \qquad
\begin{gathered}
  \phi^5(\inc(P))=5, \quad \phi^{41}(\inc(P))=3,\\
  \phi^{32}(\inc(P))=7, \quad \phi^{221}(\inc(P))=1,\\
  \phi^{311}(\inc(P))=
  \phi^{2111}(\inc(P))=
  \phi^{11111}(\inc(P))= 0
  \end{gathered}
\end{equation*}
are not consistent with
the number of
standard, cyclically row-semistrict
$P$-tableaux of shape $32$
\begin{equation*}
  \tableau[scY]{4,5|1,3,2}\,\raisebox{-4mm},\quad
  \tableau[scY]{4,5|2,1,3}\,\raisebox{-4mm},\quad
  \tableau[scY]{5,4|2,1,3}\,\raisebox{-4mm},\quad
  \tableau[scY]{5,4|3,2,1}\,\raisebox{-4mm},
\end{equation*}
or the number of standard,
$P$-record-free 
$P$-tableaux of shape $32$
\begin{equation*}
  \tableau[scY]{4,5|1,2,3}\,\raisebox{-4mm},\quad
  \tableau[scY]{5,4|1,2,3}\,\raisebox{-4mm},\quad
  \tableau[scY]{4,5|2,1,3}\,\raisebox{-4mm},\quad
  \tableau[scY]{5,4|2,1,3}\,\raisebox{-4mm},\quad
  \tableau[scY]{5,4|3,2,1}\,\raisebox{-4mm}.
\end{equation*}
%

The author has found that for $n \leq 5$,
the sets of analogous tableaux for $n$-element posets
have cardinalities no greater than the true values of $\phi^\lambda(\inc(P))$.
This suggests the following question.
\begin{quest}
  Do we have for all unit interval orders $P$
  and all partitions $\lambda \vdash |P|$,
  that $\phi^\lambda(\inc(P))$ is greater than or equal to
  \begin{enumerate}
  \item the number of standard, cyclically row-semistrict $P$-tableaux of shape $\lambda$?
  \item the number of standard, $P$-record-free $P$-tableaux of shape $\lambda$?
  \end{enumerate}
\end{quest}

\ssec{Fundamental expansion of $X_{\inc(P)}$}

We remark that for any
graph $G$,
there are known combinatorial interpretations
for the coefficients arising in the fundamental expansions of
$X_G$ and $\omega X_G$.
These are easiest to express in the special case that $G$
is the incomparability graph of an $n$-element poset $P$.
Writing
\begin{equation*}
  \begin{aligned}
    X_{\inc(P)} &= \sum_{S \subseteq [n-1]} \xi^S(\inc(P)) F_{n,[n-1]\ssm S},\\
    \omega X_{\inc(P)} &= \sum_{S \subseteq [n-1]} \xi^S(\inc(P)) F_{n,S},
  \end{aligned}
\end{equation*}
we have that $\xi^S(\inc(P))$ is the number of $P$-permutations
with $P$-descent set $S$~\cite[Cor.\,2]{ChowDes}.
For
combinatorial interpretations corresponding to an arbitrary
graph $G$, see \cite[Cor.\,1]{ChowDes}.



\ssec{Trace identities}

Symmetric function identities,
Lemma~\ref{l:efactor}, and
the combinatorial interpretations stated in
Equation~(\ref{eq:epsilonlambdaP}) -- Theorem~\ref{t:psilambdaP}
lead to some identities relating $P$-permutations
to pairs of
subposet permutations.
For instance, the first identity in the following result implies that
the number of ways to create a row-semistrict $P$-permutation
and circle one element
equals the number of
$P$-permutations $w_1 \cdots w_n$ with
$w_1 \cdots w_i$ 
a cyclically row-semistrict $P_J$-permutation
for some $i$-element set $J \subseteq [n]$ ($0 < i \leq n$),
and $w_{i+1} \cdots w_n$
a row-semistrict $P_{\tnsp \olj\,}$-permutation.
  \begin{cor}\label{c:traceid}
  Let $G$ be a graph on $n$ vertices. We have
  \begin{equation*}
    \begin{gathered}
    n \eta^n(G) = \sum_{i=1}^n \sumsb{J \subseteq [n]\\|J|=i}
    \psi^i(G_{\ntnsp J})\eta^{n-i}(G_{\olj}), \\
    n \epsilon^n(G) = \sum_{i=1}^n \sumsb{J \subseteq [n]\\|J|=i}
    (-1)^{i-1} \psi^i(G_{\ntnsp J})\epsilon^{n-i}(G_{\olj}),\\
    \sum_{i=0}^n \sumsb{J \subseteq [n]\\|J|=i}
    (-1)^i \epsilon^i(G_{\ntnsp J})\eta^{n-i}(G_{\olj}) = 0.
    \end{gathered}
  \end{equation*}
  \end{cor}
  \begin{proof}
    Applying Lemma~\ref{l:efactor} (2) to the symmetric function identities
    \begin{equation*}
        nh_n = \sum_{i=1}^n p_i h_{n-i}, \qquad
        ne_n = \sum_{i=1}^n (-1)^{i-1}p_i e_{n-i}, \qquad
        \sum_{i=0}^n e_i h_{n-i} = 0,
    \end{equation*}
    we obtain the claimed graph identities.
  \end{proof}
  
\section{Applications to total nonnegativity}\label{s:tnn}

Nonnegative expansions of chromatic symmetric functions in
the standard bases
are closely
related to
functions of totally nonnegative matrices.
We will make this relationship precise in Corollary~\ref{c:tnnimplications}.

Call a real $n \times n$ matrix $A = (a_{i,j})$ {\em totally nonnegative} if
for each pair $(I,J)$ of subsets of $[n]$,
the square submatrix $A_{I,J} \defeq (a_{i,j})_{i\in I, j\in J}$
satisfies $\det(A_{I,J}) \geq 0$. 
Such matrices are closely related to directed graphs called planar networks.
Define a (nonnegative weighted)
{\em planar network of order $n$} to be a directed, planar, acyclic
digraph $D = (V,E)$
which can be embedded in a disc so that $2n$ distinguished vertices
labeled clockwise as $s_1,\dotsc,s_n,t_n,\dotsc,t_1$
lie on the boundary of the disc,
with a nonnegative real {\em weight} $c_{u,v}$
assigned to each edge $(u,v) \in E$.
We may assume that $s_1,\dotsc,s_n$, called {\em sources}, have indegree $0$
and that $t_n,\dotsc,t_1$, called {\em sinks}, have outdegree $0$.
To every source-to-sink path, we associate a weight equal to the product
of weights of its edges, and we define the {\em path matrix}
$A = A(D) = (a_{i,j})_{i,j \in [n]}$ by
setting $a_{i,j}$ equal to the sum of weights of all paths from $s_i$ to $t_j$.
For example, we have the following planar network $D$ of order $3$,
in which unlabeled edges have weight $1$,
and its path matrix $A$.
\begin{equation}\label{eq:planarnet}
    D = \ntnsp
\begin{tikzpicture}[scale=.7,baseline=-5]
\draw[fill] (0,2) circle (1mm); \node at (-.5,2) {$s_3$};
\draw[fill] (0,0) circle (1mm); \node at (-.5,0) {$s_2$};
\draw[fill] (0,-2) circle (1mm); \node at (-.5,-2) {$s_1$};
\draw[fill] (3,2) circle (1mm); \node at (3,1.6) {$v_3$};
\draw[fill] (2,0) circle (1mm); \node at (2,-.35) {$v_2$};
\draw[fill] (2,-2) circle (1mm); \node at (2,-2.35) {$v_1$};
\draw[fill] (4,0) circle (1mm); \node at (3.85,0.4) {$v_4$};
\draw[fill] (6,2) circle (1mm); \node at (6.5,2) {$t_3$};
\draw[fill] (6,0) circle (1mm); \node at (6.5,0) {$t_2$};
\draw[fill] (6,-2) circle (1mm); \node at (6.5,-2) {$t_1$};
\draw[->,thick] (0,2) -- (2.85,2);
\draw[->,thick] (0,2) -- (1.88,0.07);
\node at (4.8,1.25) {$2$};
\node at (1.2,-.75) {$3$};
\node at (2.7,-.75) {$4$};
\node at (3,.35) {$5$};
\node at (4.5,2.35) {$6$};
\node at (4.5,-1.15) {$7$};
\draw[->,thick] (0,0) -- (1.85,0);
\draw[->,thick] (0,0) -- (1.88,-1.93);
\draw[->,thick] (0,-2) -- (1.85,-2);
\draw[->,thick] (3,2) -- (5.85,2);
\draw[->,thick] (2,0) -- (3.85,0);
\draw[->,thick] (4,0) -- (5.88,1.93);
\draw[->,thick] (4,0) -- (5.85,0);
\draw[->,thick] (2,-2) -- (3.88,-.07);
\draw[->,thick] (2,-2) -- (5.88,-.07);
\draw[->,thick] (2,-2) -- (5.85,-2);
\end{tikzpicture},
\qquad
A = \begin{bmatrix}
  1 & 11 & 8 \\
  3 & 38 & 34 \\
  0 & 5 & 16
\end{bmatrix}\ntnsp.
\end{equation}

A result often attributed to Lindstr\"om
~\cite{LinVrep} but proved earlier by Karlin and McGregor~\cite{KMG}
asserts the total nonnegativity of such a matrix.
\begin{thm}\label{t:lin}
  The path matrix $A$ of a nonnegative weighted planar network $D$ of order $n$
  is totally nonnegative.  Moreover,
  the nonnegative number $\det(A)$ equals
  \begin{equation*}
    \sum_\pi \wgt(\pi),
  \end{equation*}
  where the sum is over all families $\pi = (\pi_1,\dotsc,\pi_n)$
  of pairwise nonintersecting paths in $D$,
  with $\pi_i$ a path from $s_i$ to $t_i$ for $i = 1,\dotsc,n$, and where
  \begin{equation}\label{eq:wgt}
    \wgt(\pi) \defeq \wgt(\pi_1) \cdots \wgt(\pi_n).
    \end{equation}
\end{thm}
Since each submatrix of a totally nonnegative matrix is itself totally
nonnegative, this result gives a combinatorial interpretation of
the nonnegative numbers $\det(A_{I,J})$
as well:
$\det(A_{I,J})$ is the sum of weights of all nonintersecting path families
in $D$ from sources indexed by $I$ to sinks indexed by $J$
(assuming $|I| = |J|$).
For example, in (\ref{eq:planarnet}), we have $\det(A) = 30$,
and the one nonintersecting path family in $D$
from $\{s_1, s_2, s_3\}$ to $\{t_1, t_2, t_3 \}$
has weight
$(1\cdot 1)(1 \cdot 5 \cdot 1)(1 \cdot 6) = 30$.
Also, we have $\det(A_{13,23}) = 136$ and the three nonintersecting path families
from $\{s_1,s_3\}$ to $\{t_2, t_3\}$ have weights
$(1 \cdot 7)(1 \cdot 5 \cdot 2)$,
$(1 \cdot 7)(1 \cdot 6)$, and
$(1 \cdot 4 \cdot 1)(1 \cdot 6)$, which sum to $136$.

The converse of Theorem~\ref{t:lin} is true as well.
That is,
path matrices are essentially the only examples of totally nonnegative
matrices
~\cite{BrentiCTP}, \cite{CryerProp}, \cite{loewner}, \cite{WReduction}.
%
\begin{thm}\label{t:tnnconv}
  For each $n \times n$ totally nonnegative matrix $A$, there exists
  a nonnegative weighted planar network $D$ of order $n$
  whose path matrix is $A$.
\end{thm}

Generalizing the determinant are
matrix functions
$\imm\theta: \mat n (\mathbb C) \rightarrow \mathbb C$
called {\em immanants} and parametrized by linear functionals
$\theta: \csn \rightarrow \mathbb C$.
Define
\begin{equation*}
  \imm\theta(A) = \sum_{w \in \sn} \theta(w) a_{1,w_1} \cdots a_{n,w_n}.
\end{equation*}
In the language of Section~\ref{s:sftrace},
we have $\det(A) = \imm{\epsilon^n}(A)$.
For some functions $\theta$, the number $\imm\theta(A)$ is nonnegative
for all totally nonnegative matrices $A$ and
has a nice
combinatorial interpretation in terms of families of paths in $D$.
We say that a path family $\pi = (\pi_1,\dotsc,\pi_n)$ in a planar network $D$
has {\em type} $w = w_1 \cdots w_n \in \sn$ if for $i = 1, \dotsc, n$, the path
$\pi_i$ begins at the source $s_i$ and terminates at the sink $t_{w_i}$.
Define the sets
\begin{equation*}
  \begin{aligned}
    \mcp_w(D) &= \{ \pi \text{ in } D \,|\, \type(\pi) = w \},\\
    \mcp(D) &= \ntnsp\bigcup_{w \in \sn}\ntnsp \mcp_w(D).
  \end{aligned}
  \end{equation*}
Each path family $\pi = (\pi_1,\dotsc,\pi_n) \in \mcp_w(D)$
forms a poset $P = P(\pi)$ defined by $\pi_i <_P \pi_j$
if $i < j$ (as integers) and $\pi_i$ does not intersect $\pi_j$.
For example, the planar network $D$ in (\ref{eq:planarnet}) has two
path families of type $132$.  These and their posets are
\begin{equation}\label{eq:pirho}
\begin{tikzpicture}[scale=.7,baseline=-5]
\draw[gray,-,thin,densely dotted] (0,1) -- (3,1);
\draw[gray,-,thin,densely dotted] (0,0) -- (1,0);
\draw[-,thick,dashed] (0,-1) -- (3,-1);
\draw[gray,-,thin,densely dotted] (0,0) -- (1,-1) -- (3,0);
\draw[-,ultra thick] (0,0) -- (1,-1) -- (3,1);
\draw[-,ultra thick, dotted] (0,1) -- (1,0) -- (3,0);
  \node at (-.5,1) {$\pi_3$};  
  \node at (-.5,0) {$\pi_2$};  
  \node at (-.5,-1) {$\pi_1$};  
\end{tikzpicture}\,,
\quad \,
P(\pi) = \ntnsp
\begin{tikzpicture}[scale=.7,baseline=-5]
\draw[fill] (.8,.5) circle (1mm); \node at (.8,1) {$\pi_3$};
\draw[fill] (.8,-.5) circle (1mm); \node at (.8,-1) {$\pi_1$};
\draw[fill] (1.6,0) circle (1mm); \node at (1.6,-.5) {$\pi_2$};
\draw[-,thick] (.8,.5) -- (.8,-.5);
\end{tikzpicture}\ntksp,
\qquad \qquad
\begin{tikzpicture}[scale=.7,baseline=-5]
\draw[gray,-,thin,densely dotted] (0,1) -- (3,1);
\draw[-,ultra thick] (0,-.035) -- (2,-.035) -- (3,1);
\draw[-,thick,dashed] (0,-1) -- (3,-1);
\draw[gray,-,thin,densely dotted] (0,0) -- (1,-1) -- (3,0);
\draw[gray,-,thin,densely dotted] (1,-1) -- (3,1);
\draw[-,ultra thick, dotted] (0,1) -- (1,0.035) -- (3,0.035);
  \node at (-.5,1) {$\rho_3$};  
  \node at (-.5,0) {$\rho_2$};  
  \node at (-.5,-1) {$\rho_1$};  
\end{tikzpicture}\,,
\quad \,
P(\rho) = \ntksp
\begin{tikzpicture}[scale=.7,baseline=-5]
\draw[fill] (0,.5) circle (1mm); \node at (0,1) {$\rho_2$};
\draw[fill] (.4,-.5) circle (1mm); \node at (.4,-1) {$\rho_1$};
\draw[fill] (.8,.5) circle (1mm); \node at (.8,1) {$\rho_3$};
\draw[-,thick] (0,.5) -- (.4,-.5);
\draw[-,thick] (.8,.5) -- (.4,-.5);
\end{tikzpicture}\ntksp,
\qquad
\end{equation}
so $\mcp_{132}(D) = \{\pi,\rho\}$.

Observe that if path families $\pi$, $\sigma$ in $D$ consist of
the same multiset $K$ of edges of $D$, then they satisfy
$\wgt(\pi) = \wgt(\sigma)$.  
Call such a multiset $K$ a {\em bijective skeleton},
and define $\wgt(K)$ to be the product of its edge weights, with multiplicities.
Define the sets
\begin{equation*}
  \begin{aligned}
    \Pi(K) &= \{ \pi \in \mcp(D) \,|\, \text{edge multiset of $\pi$ is $K$}\},\\
    \Pi_w(K) &= \{ \pi \in \Pi(K) \,|\, \type(\pi) = w\},
  \end{aligned}
\end{equation*}
and the $\zsn$-element
\begin{equation*}
  z(K) = \ntksp \sum_{\pi \in \Pi(K)} \ntksp \type(\pi).
\end{equation*}
For example, let bijective skeleton $K$ be the multiset of edges of $D$
(\ref{eq:planarnet}) covered the path family $\rho$ in (\ref{eq:pirho}),
and let $\sigma$ be the unique path family of type $123$ covering $K$.
Then we have
\begin{equation*}
  K =
\begin{tikzpicture}[scale=.7,baseline=-5]
\draw[gray,-,thin,densely dotted] (0,1) -- (3,1);
\draw[-,ultra thick] (0,0) -- (2,0) -- (3,1);
\draw[-,ultra thick] (0,-1) -- (3,-1);
\draw[gray,-,thin,densely dotted] (0,0) -- (1,-1) -- (3,0);
\draw[gray,-,thin,densely dotted] (1,-1) -- (3,1);
\draw[-,ultra thick] (0,1) -- (1,0) -- (3,0);
  \node at (1.5,.3) {$_{(2)}$};  
\end{tikzpicture}\,,
\qquad \qquad
\begin{tikzpicture}[scale=.7,baseline=-5]
\draw[gray,-,thin,densely dotted] (0,1) -- (3,1);
\draw[-,ultra thick] (0,-.035) -- (3,-.035);
\draw[-,thick,dashed] (0,-1) -- (3,-1);
\draw[gray,-,thin,densely dotted] (0,0) -- (1,-1) -- (3,0);
\draw[gray,-,thin,densely dotted] (1,-1) -- (3,1);
\draw[-,ultra thick, dotted] (0,1) -- (1,.035) -- (2, .035) -- (3,1);
  \node at (-.5,1) {$\sigma_3$};  
  \node at (-.5,0) {$\sigma_2$};  
  \node at (-.5,-1) {$\sigma_1$};  
\end{tikzpicture}\,,
\end{equation*}
with $(2)$ marking the edge in $K$ having multiplicity $2$.
We also have
\begin{equation*}
  \begin{gathered}
\Pi_{132}(K) = \{\rho\},\qquad
\Pi_{123}(K) = \{\sigma\},\qquad
\Pi(K) = \{\rho, \sigma\},\\
\wgt(K) = \wgt(\rho) = \wgt(\sigma) = 50, \qquad
z(K) = 132 + 123 = \wtc{132}1.
  \end{gathered}
  \end{equation*}
It is known that for any bijective skeleton $K$, 
$z(K)$ equals a product of Kazhdan-Lusztig
basis elements $\wtc w1 \in \zsn$ indexed by $312$-avoiding permutations,
and that
for any totally nonnegative matrix $A$, the numbers
$\imm{\theta}(A)$ and $\theta(z(K))$ are closely related.
(See, e.g., \cite[Thm.\,2.1]{StemImm}.)
\begin{prop}\label{p:skeleton}
  Let $A$ be the path matrix of a weighted planar network $D$.
  Then for any linear functional $\theta: \csn \rightarrow \mathbb C$, we have
  \begin{equation*}
    \imm \theta(A) = \sum_K \wgt(K) \theta(z(K)),
  \end{equation*}
  where the sum is over all bijective skeletons $K$ in $D$.
\end{prop}
\begin{proof}
  By the definition of path matrix,
  we can interpret
  each product of matrix entries
  appearing in
  $\imm \theta(A)$
  as
  \begin{equation*}
  a_{1,w_1} \cdots a_{n,w_n} = 
  \sum_K \ntnsp \sum_{\pi \in \Pi_w(K)} \ntksp
  \wgt(\pi) = \sum_K \wgt(K) | \Pi_w(K) |.
  \end{equation*}
  Multiplying each product by $\theta(w)$, summing over $w \in \sn$,
  and using the linearity of $\theta$,
  we may thus express $\imm \theta(A)$ as
  \begin{equation*}
    \begin{aligned}
    \sum_{w \in \sn} \ntnsp \theta(w) \sum_K \wgt(K) |\Pi_w(K)|
    &= \sum_K \wgt(K) \ntnsp \sum_{w \in \sn} \ntnsp \theta(w) |\Pi_w(K)| \\
    &= \sum_K \wgt(K) \,\theta \Big( \ntksp \sum_{w \in \sn} \ntksp |\Pi_w(K)| w \Big ) \\
    &= \sum_K \wgt(K) \,\theta \Big( \ntksp\ntnsp \sum_{\pi \in \Pi(K)} \ntksp\ntnsp \type(w) \Big ).
    \end{aligned}
  \end{equation*}
\end{proof}
Sometimes a combinatorial interpretation for $\imm{\theta}(A)$ comes
from careful consideration of $\theta(z(K))$; other times it comes
from a simple expression for $\imm{\theta}(A)$, such as the
Littlewood-Merris-Watkins
identities~\cite[\S 6.5]{LittlewoodTGC}, \cite[\S 1]{MerWatIneq},
  \begin{gather}
    \imm{\epsilon^\lambda}(A) =
    \ntksp \sum_{(I_1,\dotsc,I_r)} \ntksp
    \det(A_{I_1,I_1}) \cdots \det(A_{I_r,I_r}),\label{eq:lmw1}\\
    \imm{\eta^\lambda}(A) =
    \ntksp \sum_{(I_1,\dotsc,I_r)} \ntksp
    \perm(A_{I_1,I_1}) \cdots \perm(A_{I_r,I_r})\label{eq:lmw2},
  \end{gather}
  where the sums are over ordered set partitions of $[n]$ of type
  $\lambda = (\lambda_1,\dotsc,\lambda_r).$
  
\ssec{Induced sign character immanants}

Combinatorial interpretations for the immanants $\imm{\epsilon^\lambda}(A)$
follow easily from Theorem~\ref{t:lin}
and (\ref{eq:lmw1}).
\begin{thm}\label{t:epsilonimm}
  Let planar network $D$ have path matrix $A$.  Then we have
  \begin{equation}\label{eq:epsilonimm}
  \imm{\epsilon^\lambda}(A) = \sum_K \wgt(K) \nTksp \sum_{\pi \in \Pi_e(K)} \nTksp 
  \epsilon^\lambda(\inc(P(\pi))).
  \end{equation}
\end{thm}
\begin{proof}
  By Theorem~\ref{t:lin} and the comment immediately following it,
  the term of (\ref{eq:lmw1})
  corresonding to a fixed ordered set partition $(I_1,\dotsc,I_r)$
  is equal to the sum of weights of path families
  $\pi = (\pi_1,\dotsc,\pi_n)$ of type $e$ in which for $j=1,\dotsc,r$,
  paths indexed by $I_j$ are pairwise nonintersecting.
  This partitioned path family naturally forms a column-strict tableau
  $U = U(\pi,I_1,\dotsc,I_r)$ of shape
  $\lambda^\tr$, if we place paths indexed by $I_j$ into column $j$.
  We
  may therefore write the right-hand-side of (\ref{eq:lmw1}) as
  \begin{equation*}
    \begin{aligned}
      &\sum_K \wgt(K)
      \nTksp \sum_{\pi \in \Pi_e(K)} \nTksp \#
      \{ (I_1,\dotsc,I_r) \,|\, U(\pi,I_1,\dotsc,I_r)
      \text{ is column-strict of shape $\lambda^\tr$\,} \}\\
      &\qquad = \sum_K \wgt(K)
      \nTksp \sum_{\pi \in \Pi_e(K)} \nTksp \#
      \text{ column-strict $P(\pi)$-tableaux of shape $\lambda^\tr$ }\\
      &\qquad = \sum_K \wgt(K)
      \nTksp \sum_{\pi \in \Pi_e(K)} \nTksp 
      \epsilon^\lambda(\inc(P(\pi))).
      \end{aligned}
  \end{equation*}
\end{proof} 

Combining Proposition~\ref{p:skeleton} with Theorem~\ref{t:epsilonimm},
we find that the inner sum of (\ref{eq:epsilonimm}) is $\epsilon^\lambda(z(K))$
and we generalize this fact as follows.
\begin{cor}\label{c:skeleton}
  Let $K$ be a bijective skeleton in a planar network $D$.
  Then for all traces
  $\theta: \csn \rightarrow \mathbb C$, we have
  \begin{equation*}
    \theta(z(K)) = \ntksp \sum_{\pi \in \Pi_e(K)} \ntksp \theta(\inc(P(\pi))).
  \end{equation*}
\end{cor}
\begin{proof}
  Weight the planar network $D$ by algebraically
  independent real numbers and let $A$ be its path matrix.
  By Proposition~\ref{p:skeleton}, we have
  \begin{equation}\label{eq:epsiloninterp2}
    \imm{\epsilon^\mu}(A) = \sum_K \wgt(K) \epsilon^\mu(z(K)),
  \end{equation}
  where the sum is over all bijective skeletons $K$ in $D$.
  Since the edge weights of $D$ are algebraically independent,
  we may compare this expression to the right-hand-side of
  (\ref{eq:epsilonimm}) to obtain
  \begin{equation*}
    \epsilon^\mu(z(K)) = \nTksp \sum_{\pi \in \Pi_e(K)} \nTksp
    \epsilon^\mu(\inc(P(\pi))).
  \end{equation*}
  Expanding $\theta$ in the induced sign character basis
  $\{ \epsilon^\mu \,|\, \mu \vdash n\}$, we obtain the desired result.
\end{proof}

Proposition~\ref{p:skeleton} and
Corollary~\ref{c:skeleton} show that
for $\theta \in \trspace n$
we may compute $\imm \theta(A)$ by considering
a planar network $D$ having path matrix $A$,
each path family $\pi$ of type $e$ in $D$,
and the corresponding
chromatic symmetric function $X_{\inc(P(\pi))}$.
\begin{cor}
  For $D$ a planar network having path matrix $A$, we have
\begin{equation}\label{eq:propcorcomb}
  \imm{\theta}(A) = \sum_K \wgt(K)
  \nTksp \sum_{\pi \in \Pi_e(K)} \nTksp \theta(\inc(P(\pi)),
\end{equation}
where $K$ varies over all bijective skeletons in $D$.
\end{cor}
Thus if $\theta \in \trspace n$ satisfies
$\theta(\inc(P)) \geq 0$ for all posets $P$, then it also satisfies
$\imm\theta(A) \geq 0$
for all totally nonnegative matrices $A$.
For the convenience of the reader we summarize this
and other known implications as follows.
\begin{cor}\label{c:tnnimplications}
  For $\theta \in \trspace n$, the statements
  \begin{enumerate}
  \item $\theta( \wtc w1) \geq 0$ for all permutations $w \in \sn$,
  \item $\theta( \wtc{w^{(1)}}1 \cdots \wtc{w^{(k)}}1) \geq 0$ for all sequences
   $(w^{(1)},\dotsc,w^{(k)})$ of maximal elements of parabolic subgroups of $\sn$,
  \item $\theta(\inc(P)) \geq 0$ for all posets $P$,
  \item $\imm\theta(A) \geq 0$ for all totally nonnegative matrices $A$,
  \item $\theta(\inc(P)) \geq 0$ for all
    unit interval orders $P$,
  \item $\theta(\inc(P)) \geq 0$ for all
    $(\mathbf3 + \mathbf1)$-free posets $P$,
  \item $\theta(\wtc w1) \geq 0$ for all $312$-avoiding permutations $w \in \sn$,
  \item $\theta(\wtc w1) \geq 0$ for all \pavoiding permutations $w \in \sn$
\end{enumerate}
  satisfy the implications (1) $\Rightarrow$ (2) $\Rightarrow$ (4)
  $\Rightarrow$ (5) $\Leftrightarrow$ (6) $\Leftrightarrow$ (7) $\Leftrightarrow$ (8), and
  (3) $\Rightarrow$ (4).
\end{cor}
\begin{proof}

  \noindent
  ($(1) \Rightarrow (2) \Rightarrow (4)$)
  Suppose that $(4)$ is false, i.e., that
  for some totally nonnegative matrix $A$ we have
  $\imm\theta(A) < 0$.
  By Theorem~\ref{t:tnnconv}, $A$ is the path matrix of a planar network $D$.
  Then Proposition~\ref{p:skeleton} implies that
  some bijective skeleton $K$ of $D$ satisfies $\theta(z(K)) < 0$.
  It is straightforward to show that $z(K)$
  equals a (positive rational multiple of a)
  product of Kazhdan-Lusztig basis elements of the form
  $\wtc{w^{(1)}}1 \cdots \wtc{w^{(k)}}1$
  in which each permutation $w^{(i)}$ is a maximal element of a parabolic
  subgroup of $\sn$. (See, e.g., \cite[Cor.\,5.3]{CSkanTNNChar}.)
  Thus $(2)$ is false.
  This product in turn equals a nonnegative linear combination of
  Kazhdan-Lusztig basis elements. (See \cite[Appendix]{HaimanHecke}.)
  Thus $\theta(\wtc w1) < 0$ for some $w \in \sn$ and $(1)$ is false.

  \noindent ($(4) \Rightarrow (5)$)
  For every unit interval order $P$, there exists a planar network
  $D(P)$ with path matrix $A$ satisfying
  $\theta(\inc(P)) = \imm\theta(A)$~\cite[Prop.\,3.8, Thm.\,4.1, Cor.\,7.5]{CHSSkanEKL}.

  \noindent ($(5) \Leftrightarrow (6)$)
  Each unit interval order is
  $(\mathbf3 + \mathbf 1)$-free.  Conversely, 
  for each $(\mathbf3 + \mathbf1)$-free poset $P$, we have by
  \cite[Thm.\,5.3]{GPModRel} that $X_{\inc(P)}$ belongs to the cone
  generated by chromatic symmetric functions of unit interval orders.

  \noindent ($(5) \Rightarrow (8)$)
  By \cite[Thm.\,3.5, Lem.\,5.3]{SkanNNDCB}
  and \cite[Thm.\,7.4]{CHSSkanEKL}, we have that 
  for each \pavoiding permutation $w \in \sn$ there exists a
  unit interval order $P = P(w)$ satisfying
  $\theta(\inc(P(w)) = \theta(\wtc w1)$ for all
  $\theta \in \trspace n$.

  \noindent ($(8) \Rightarrow (7) \Rightarrow (5)$)
  Each $312$-avoiding permutation also \avoidsp.
  The bijection (\ref{eq:uioto312avoid}) to unit interval orders
  and Proposition~\ref{p:uioto312avoid} give the last implication.
  %
  
  \noindent ($(3) \Rightarrow (4)$) Follows from (\ref{eq:propcorcomb}).
\end{proof}  

Another consequence of Theorem~\ref{t:epsilonimm} is an analog of
Lemma~\ref{l:efactor} (2) for matrices. (See also \cite[Prop.\,2.4]{StemConj}.)
\begin{cor}\label{c:immfactor}
  For $\theta_1 \in \trspace k$, $\theta_2 \in \trspace{n-k}$,
  $\theta = \theta_1 \otimes \theta_2 \upparrow_{\mfs k \times \mfs{n-k}}^{\sn} \in \trspace n$,
  we have
  \begin{equation*}
    \imm\theta(A) = \sum_{J \subseteq [n]}
    \imm{\theta_1}(A_{J,J}) \imm{\theta_2}(A_{\tnsp\olj, \tnsp\olj}).
  \end{equation*}
\end{cor}
\begin{proof}
  Similar to proof of Lemma~\ref{l:efactor} (2).  
\end{proof}

\ssec{Irreducible character immanants}

While no combinatorial interpretation is known for
$\imm{\chi^\lambda}(A)$, Stembridge~\cite[Cor.\,3.3]{StemImm}
proved the following.
\begin{thm}\label{t:stemimm}
  For $\lambda \vdash n$
  and $A$ an $n \times n$ totally nonnegative
  matrix we have $\imm{\chi^\lambda}(A) \geq 0$.
\end{thm}
\bp\label{p:stemimm}
Combinatorially interpret the numbers $\imm{\chi^\lambda}(A)$ in
Theorem~\ref{t:stemimm}.
\ep
We do have a combinatorial interpretation in the special case that
$\lambda$ is a hook shape.
\begin{thm}\label{t:hookimm}
  For $A$ the path matrix of planar network $D$ of order $n$ and $k \leq n$,
  we have
  \begin{equation}\label{eq:hookimm}
    \imm{\chi^{k1^{n-k}}}(A) = \nTksp \sum_{\pi \in \mcp_e(D)} \nTksp \wgt(\pi)\,
    (\# \text{ standard $P(\pi)$-tableaux of shape $k1^{n-k}$}).
  \end{equation}
  In particular, when $k=n$, we obtain
  (\ref{eq:introdes}).
\end{thm}
\begin{proof}
  Let $\lambda = k1^{n-k}$.
  By Proposition~\ref{p:skeleton}, Corollary~\ref{c:skeleton},
  and Proposition~\ref{p:Kali},
  we have
  \begin{equation*}
    \begin{aligned}
    \imm{\chi^\lambda}(A) &= \sum_K \wgt(K) \chi^\lambda(z(K))\\
    &= \sum_K \wgt(K) \nTksp \sum_{\pi \in \Pi_e(K)} \nTksp \chi^\lambda(P(\pi))\\
    &= \sum_K \wgt(K) \nTksp \sum_{\pi \in \Pi_e(K)} \nTksp
    \# \tnsp \text{standard $P(\pi)$-tableaux of shape $\lambda$},
    \end{aligned}
  \end{equation*}
  where $K$ varies over all bijective skeletons in $D$.
  This is equal to the claimed expression.
\end{proof}
The case $k=n$ (\ref{eq:introdes}) can also be deduced
from Stanley's interpretation~\cite[Thm.\,3.3]{StanSymm}
of $\chi^n(\inc(P))$ as the
number of acyclic orientations of $\inc(P)$,
using the bijection at the end of
the proof of Theorem~\ref{t:etalambdaP}.

Returning to Corollary~\ref{c:tnnimplications},
we see that for all $\lambda \vdash n$,
Haiman's result~\cite[Lem.\,1.1]{HaimanHecke}
that $\chi_q^\lambda(\wtc wq) \in \mathbb N[q]$ for
all
$w \in \sn$
implies Stembridge's result
~\cite[Cor.\,3.3]{StemConj}
that $\imm{\chi^\lambda}(A) \geq 0$
for all $A$ totally nonnegative,
which in turn implies Gasharov's result~\cite{GashInc}
that $\chi^\lambda(\inc(P)) \geq 0$
for all $(\mathbf3 + \mathbf1)$-free posets $P$.
The failure of the inequality
$\chi^\lambda(\inc(P)) \geq 0$ to hold for all posets $P$
and the equation (\ref{eq:propcorcomb}) suggest
that one might solve Problem~\ref{p:stemimm}
by explaining why
for each bijective skeleton $K$ in a planar network,
we have
\begin{equation*}
  \sum_{\pi \in \Pi_e(K)} \nTksp \chi^\lambda(\inc(P(\pi)) \geq 0,
  \qquad \text{equivalently,} \qquad
  \sum_{\pi \in \Pi_e(K)} \nTksp \ntnsp X_{\inc(P(\pi))} \text{ is Schur-positive},
\end{equation*}
even when some of the posets $\{P(\pi) \,|\, \pi \in \Pi_e(K)\}$
are not ($\mathbf3 + \mathbf1$)-free.

\ssec{The permanent and induced trivial characters}\label{ss:perm}


An obvious consequence of
Proposition~\ref{p:skeleton}
is a combinatorial interpretation of
the permanent of a totally nonnegative matrix.

\begin{obs}\label{o:perm}
  Let totally nonnegative matrix $A$ be the path matrix of
  planar network $D$.  Then we have
  \begin{equation*}
    \perm(A) =
    \ntksp \sum_{\pi \in \mcp(D)} \ntksp \wgt(\pi),
  \end{equation*}
  where
  $\wgt(\pi)$ is defined as in (\ref{eq:wgt}).
\end{obs}
Two more combinatorial interpretations
make use of the partial orders $\{P(\pi) \,|\, \pi \in \mcp_e(D) \}$.
We have shown in Theorem~\ref{t:hookimm} and will show in
Theorem~\ref{t:excfree} that
  \begin{align}
    \perm(A)
    &= \nTksp \sum_{\pi \in \mcp_e(D)} \nTksp
    \wgt(\pi) \, (\#\text{$P(\pi)$-descent-free permutations of $\pi$})
    \label{eq:introdes}\\
    &= \nTksp \sum_{\pi \in \mcp_e(D)} \nTksp
    \wgt(\pi) \, (\#\text{$P(\pi)$-excedance-free permutations of $\pi$}).
    \label{eq:introexc}
  \end{align}
  Unfortunately we do not know how to
  obtain bijective proofs of these facts
  directly from Observation~\ref{o:perm}.
\bp\label{p:permbijection}
Given a planar network $D$ of order $n$, state an explicit bijection between
path families in $D$ of arbitrary type,
and $P$-descent-free $P(\pi)$-permutations, where $\pi$ in $D$ has type $e$,
or $P$-excedance-free $P(\pi)$-permutations, where $\pi$ in $D$ has type $e$,
\begin{equation*}
  \begin{gathered}
  \mcp(D)
  \qquad
  \overset{1-1}\longleftrightarrow
  \quad 
  \bigcup_{\pi \in \mcp_e(D)} \ntksp
  \{ U\, \text{a $P(\pi)$-permutation}\, \,|\, \des_{P(\pi)}(U) = 0 \},\\
  \mcp(D)
  \qquad
  \overset{1-1}\longleftrightarrow
  \quad
  \bigcup_{\pi \in \mcp_e(D)} \ntksp
  \{ U\, \text{a $P(\pi)$-permutation}\, \,|\, \exc_{P(\pi)}(U) = 0 \}.
  \end{gathered}
\end{equation*}
\ep

On the other hand,
Theorem~\ref{t:hookimm} and
Proposition~\ref{p:equidist}
easily prove (\ref{eq:introexc}).

\begin{thm}\label{t:excfree}
  Let $A$ be the path matrix of planar network $D$ of order $n$,
  Then we have
  \begin{equation*}
    \perm(A) = \sum_{\pi \in \mcp_e(D)} \wgt(\pi)
    (\# \text{ $P(\pi)$-excedance-free permutations of $\pi$}).
  \end{equation*}
\end{thm}
\begin{proof}
  Theorem~\ref{t:hookimm} implies (\ref{eq:introdes}), and
  Proposition~\ref{p:equidist} then implies (\ref{eq:introexc}).
\end{proof}

Now we have three combinatorial interpretations of induced trivial character
immanants.  To state these, we define more generalizations of Young tableaux.
Given a path family $\pi = (\pi_1,\dotsc,\pi_n) \in \mcp(D)$
for some planar network $D$,
define a {\em $\pi$-tableau of shape $\lambda$} to be a filling of a Young
diagram with $\pi_1,\dotsc,\pi_n$.
For each $\pi$-tableau $U$, define $L(U)$ and $R(U)$ to be the Young tableaux
whose integer entries are the source indices and sink indices, respectively,
of the corresponding paths in $U$.
Call the $\pi$-tableau $U$ {\em left row-strict}
if entries of $L(U)$ increase from left to right in each row,
and call it
{\em row-closed} if $R(U_i)$ is a rearrangement of $L(U_i)$ for each $i$.
For example,
consider the planar network $D$ and path family $\pi \in \mcp(D)$,
\begin{equation}\label{eq:dpi}
  D =
\begin{tikzpicture}[scale=.55,baseline=28]
  \node at (-.5,4) {$s_5$};  \node at (4.5,4) {$t_5$};
  \node at (-.5,3) {$s_4$};  \node at (4.5,3) {$t_4$};
  \node at (-.5,2) {$s_3$};  \node at (4.5,2) {$t_3$};
  \node at (-.5,1) {$s_2$};  \node at (4.5,1) {$t_2$};
  \node at (-.5,0) {$s_1$};  \node at (4.5,0) {$t_1$};
\draw[-,thick] (0,4) -- (4,0);
\draw[-,thick] (0,3) -- (1,4) -- (4,4);
\draw[-,thick] (0,2) -- (1,2) -- (2,3) -- (4,3);
\draw[-,thick] (0,1) -- (2,1) -- (3,2) -- (4,2);
\draw[-,thick] (0,0) -- (3,0) -- (4,1);
\end{tikzpicture}\,,
\qquad \qquad
  \begin{tikzpicture}[scale=.55,baseline=28]
  \node at (-.5,4) {$\pi_5$};  
  \node at (-.5,3) {$\pi_4$};  
  \node at (-.5,2) {$\pi_3$};  
  \node at (-.5,1) {$\pi_2$};  
  \node at (-.5,0) {$\pi_1$};  
\draw[-, ultra thick, dotted] (0,4) -- (1.5, 2.5) -- (2,3) -- (4,3);
\draw[-, thin] (0,3) -- (1,4) -- (4,4);
\draw[-, ultra thick] (0,2) -- (1,2) -- (1.5, 2.5) -- (4,0);
\draw[-, very thick, densely dotted] (0,1) -- (2,1) -- (3,2) -- (4,2);
\draw[-, thick, dashed] (0,0) -- (3,0) -- (4,1);
\end{tikzpicture}\,\,.
\end{equation}
In order for a $\pi$-tableau to be row-closed and left row-strict,
the paths $\pi_1, \pi_2, \pi_3$ must appear in order of increasing indices
in the same row, as must $\pi_4, \pi_5$.
All five paths must appear in order of
increasing indices if they appear in a single row.
There is one such $\pi$-tableau
of shape $5$ and one of shape $32$.
Together with their left and right Young tableaux,
these are
\begin{equation*}
  \begin{alignedat}{3}
    U \ntnsp&= {\tableau[scY]{\pi_1,\pi_2,\pi_3,\pi_4,\pi_5}}\,,
    \phantom{\Bigg|}
    &\qquad
  L(U) \ntnsp&= {\tableau[scY]{1,2,3,4,5}}\,,
  &\qquad
  R(U) \ntnsp&= {\tableau[scY]{2,3,1,5,4}}\,, \\
  V \ntnsp&= {\tableau[scY]{\pi_4,\pi_5|\pi_1,\pi_2,\pi_3}}\,,
  &\qquad
  L(V) \ntnsp&= {\tableau[scY]{4,5|1,2,3}}\,,
  &\qquad
  R(V) \ntnsp&= {\tableau[scY]{5,4|2,3,1}}\,.
  \end{alignedat}
\end{equation*}

\begin{thm}\label{t:etaimm}
  Let $A$ be the path matrix of planar network $D$ of order $n$.
  Then for $\lambda \vdash n$, the evaluation $\imm{\eta^\lambda}(A)$ has
  the combinatorial interpretations
  \begin{align}
    &\sum_{\pi \in \mcp(D)} \nTksp \wgt(\pi)\,
    (\# \text{row-closed, left row-strict $\pi$-tableaux of shape $\lambda$}),
    \label{eq:etaimm1}\\
    &\sum_{\pi \in \mcp_e(D)} \nTksp \wgt(\pi)\,
    (\# \text{descent-free $P(\pi)$-tableaux of shape $\lambda$}),
    \label{eq:etaimm2}\\
    &\sum_{\pi \in \mcp_e(D)} \nTksp \wgt(\pi)\,
    (\# \text{excedance-free $P(\pi)$-tableaux of shape $\lambda$}).
    \label{eq:etaimm3}
  \end{align}
\end{thm}
\begin{proof}
  Express $\imm{\eta^\lambda}(A)$ as in (\ref{eq:lmw2}).
  By Observation~\ref{o:perm},
  the term in this sum
  corresponding to a fixed ordered set partition $(I_1,\dotsc,I_r)$
  is equal to the sum of weights of path families
  $\pi = (\pi_1,\dotsc,\pi_n)$ in which for $j=1,\dotsc,r$,
  paths indexed by $I_j$
  have sink indices also belonging to $I_j$. 
  Each such partitioned path family $\pi$ naturally forms a left row-strict,
  row-closed $\pi$-tableau
  $U = U(\pi,I_1,\dotsc,I_r)$ of shape
  $\lambda$, if we place paths indexed by $I_j$ into row $j$,
  with path indices increasing from left to right.
  We may therefore express $\imm{\eta^\lambda}(A)$ as
  \begin{equation*}
    \sum_K \wgt(K)
    \nTksp \sum_{\pi \in \Pi(K)} \nTksp \#
    \{ (I_1,\dotsc,I_r) \,|\, U(\pi,I_1,\dotsc,I_r)
    \text{ is row-closed, left row-strict of shape $\lambda$} \},
  \end{equation*}
  i.e., as (\ref{eq:etaimm1}).
  Alternatively, by Theorems~\ref{t:hookimm} and \ref{t:excfree},
  the term in (\ref{eq:lmw2})
  corresponding to a fixed ordered set partition $(I_1,\dotsc,I_r)$
  is equal to
  \begin{equation*}
    \begin{aligned}
    &\sum_K \wgt(K) \nTksp\sum_{\pi \in \Pi_e(K)}\nTksp \#\{ (U_1,\dotsc,U_r) \,|\,
    U_j \text{ a descent-free permutation of $(\pi_i)_{i \in I_j}$} \}\\
    = &\sum_K \wgt(K) \nTksp\sum_{\pi \in \Pi_e(K)}\nTksp \#\{ (U_1,\dotsc,U_r) \,|\,
    U_j \text{ an excedance-free permutation of $(\pi_i)_{i \in I_j}$} \}.
    \end{aligned}
  \end{equation*}
  Thus $\imm{\eta^\lambda}(A)$
  is also equal to (\ref{eq:etaimm2}) and (\ref{eq:etaimm3}).
\end{proof}

\
\ssec{Power sum immanants}

Like the induced trivial character immanants
$\{ \imm{\eta^\lambda}(A) \,|\, \lambda \vdash n \}$,
the power sum immanants
$\{ \imm{\psi^\lambda}(A) \,|\, \lambda \vdash n \}$
have some combinatorial interpretations which are closely related to
chromatic symmetric function coefficients,
and others which are related to path families in a planar network.
Call a $\pi$-tableau $U$ of shape $\lambda$ {\em cylindrical}
if in each row $U_i = U_{i,1} \cdots U_{i, \lambda_i}$,
we have $R(U_{i,1} \cdots U_{i, \lambda_i}) = L(U_{i,2} \cdots U_{i,\lambda_i} U_{i,1})$,
i.e., each path begins where the preceding path in its row terminates.  
For example,
let $\pi$ be the path family in (\ref{eq:dpi}).
There are six cylindrical $\pi$-tableaux, all of shape $32$:
\begin{equation*}
  \tableau[scY]{\pi_4,\pi_5|\pi_1,\pi_2,\pi_3}\,,\quad
  \tableau[scY]{\pi_5,\pi_4|\pi_1,\pi_2,\pi_3}\,,\quad
  \tableau[scY]{\pi_4,\pi_5|\pi_2,\pi_3,\pi_1}\,,\quad
  \tableau[scY]{\pi_5,\pi_4|\pi_2,\pi_3,\pi_1}\,,\quad
  \tableau[scY]{\pi_4,\pi_5|\pi_3,\pi_1,\pi_2}\,,\quad
  \tableau[scY]{\pi_5,\pi_4|\pi_3,\pi_1,\pi_2}\,.
\end{equation*}

\begin{thm}\label{t:powerimm}
  Let $A$ be the path matrix of planar network $D$ of order $n$.
  Then for $\lambda \vdash n$, the evaluation $\imm{\psi^\lambda}(A)$
  has the combinatorial interpretations
  \begin{align}
    &\sum_{\pi \in \mcp_e(D)} \nTksp \wgt(\pi)\,
    (\# \text{cyclically row-semistrict $P(\pi)$-tableaux of shape $\lambda$}),
    \label{eq:powerimm1} \\
    &\sum_{\pi \in \mcp_e(D)} \nTksp \wgt(\pi)\,
    (\# \text{$P(\pi)$-record-free row-semistrict $P(\pi)$-tableaux
      of shape $\lambda$}),\label{eq:powerimm2}\\
    &\sum_{\pi \in \mcp(D)} \nTksp \wgt(\pi)\,
    (\# \text{cylindrical $\pi$-tableaux of shape $\lambda$}).
    \label{eq:powerimm3}
  \end{align}
\end{thm}
  \begin{proof}
  By Proposition~\ref{p:skeleton} and Corollary~\ref{c:skeleton},
  we have
  \begin{equation*}
    \imm{\psi^\lambda}(A) = \sum_K \wgt(K) \psi^\lambda(z(K))
    = \sum_K \wgt(K) \nTksp \sum_{\pi \in \Pi_e(K)} \nTksp \psi^\lambda(\inc(P(\pi))).
  \end{equation*}
  Thus by Proposition~\ref{t:psilambdaP} we have the interpretations
  (\ref{eq:powerimm1}) and (\ref{eq:powerimm2}).
  By (\ref{eq:psidef}), we have
  \begin{equation*}
    \imm{\psi^\lambda}(A) = z_\lambda
    \nTksp\ntksp \sumsb{w \in \sn\\ \ctype(w) = \lambda} \nTksp\ntksp
    a_{1,w_1} \cdots a_{n,w_n}.
  \end{equation*}
  The interpretation (\ref{eq:powerimm3})
  now follows from the definition of path matrix.
  \end{proof}
  


\ssec{Monomial immanants}

By Corollary~\ref{c:tnnimplications}, the fact that
we do not know $\phi^\lambda(\inc(P))$ to be nonnegative for unit interval orders
$P$ (\cite[Conj.\,5.5]{StanStemIJT})
implies that we do not know monomial immanants to evaluate nonnegatively
on totally nonnegative matrices.
We have the following conjecture of Stembridge~\cite[Conj.\,2.1]{StemConj}.
\begin{conj}\label{c:monimm}
  For $\lambda \vdash n$ and $A$ totally nonnegative we have
  $\imm{\phi^\lambda}(A) \geq 0$.
\end{conj}
This is the strongest possible conjecture for inequalities of the
form $\imm\theta(A) \geq 0$ with $\theta \in \trspace n$, since
  $\imm\theta(A) \geq 0$
  for all totally nonnegative $A \in \mat n (\mathbb R)$
  only if $\theta$ is a nonnegative linear combination of
  $\{ \phi^\lambda \,|\, \lambda \vdash n \}$~\cite[Prop.\,2.3]{StemConj}.
  Stembridge proved two special cases of
  Conjecture~\ref{c:monimm}~\cite[Thms.\,2.7 -- 2.8]{StemConj}.
  \begin{prop}\label{p:stemmon}
    Let $A$ be the path matrix of a planar network $D$ of order $n$.
    \begin{enumerate}
    \item If $\lambda = 21^{n-2}$ then $\imm{\phi^\lambda}(A) \geq 0$.
    \item If $\lambda$ is the rectangular shape $r^k$ then
      \begin{equation}\label{eq:stemmon}
        \imm{\phi^\lambda}(A) = \sum_K \wgt(\pi)
        \ntksp \sum_{\pi \in \Pi(K)} \ntksp
        \#\text{ column-strict, cylindrical } \pi\text{-tableaux of shape } r^k.
      \end{equation}
    \end{enumerate}
  \end{prop}

Since
$\phi^{1^n} = \epsilon^n$ and $\phi^n = \psi^n$, one might hope that
the formula (\ref{eq:stemmon})
for $\imm{\phi^\lambda}(A)$, which combines
aspects of
Theorem~\ref{t:lin} and Theorem~\ref{t:powerimm},
might hold for arbitrary $\lambda \vdash n$.
Unfortunately, it does not.
The planar network $D$ in (\ref{eq:dpi})
has path matrix
\begin{equation*}
A = \begin{bmatrix}
  1 & 1 & 0 & 0 & 0 \\
  1 & 1 & 1 & 0 & 0 \\
  1 & 1 & 1 & 1 & 0 \\
  1 & 1 & 1 & 1 & 1 \\
  1 & 1 & 1 & 1 & 1
\end{bmatrix},
\end{equation*}
which satisfies
\begin{equation*}
  \begin{gathered}
  \imm{\phi^5}(A) = 5, \qquad
  \imm{\phi^{41}}(A) = 3, \qquad
  \imm{\phi^{32}}(A) = 7, \qquad
  \imm{\phi^{221}}(A) = 1, \qquad\\
  \imm{\phi^{311}}(A) = \imm{\phi^{2111}}(A) = \imm{\phi^{11111}}(A) = 0.
  \end{gathered}
\end{equation*}
Each of these monomial immanant evaluations,
except for $\imm{\phi^{32}}(A)$,
is consistent with the combinatorial interpretation given in (\ref{eq:stemmon}).
On the other hand, the path family $\pi$ in (\ref{eq:dpi})  
is the unique path family covering $D$ which
can be placed into a column-strict cylindrical path tableau of shape $32$.
There are only four such tableaux:
\begin{equation*}
  \tableau[scY]{\pi_4,\pi_5|\pi_1,\pi_2,\pi_3}\,,\quad
  \tableau[scY]{\pi_5,\pi_4|\pi_1,\pi_2,\pi_3}\,,\quad
  \tableau[scY]{\pi_5,\pi_4|\pi_2,\pi_3,\pi_1}\,,\quad
  \tableau[scY]{\pi_4,\pi_5|\pi_3,\pi_1,\pi_2}\,.
\end{equation*}

The author has found that for $n \leq 5$,
the sets of analogous tableaux for certain planar networks
have cardinalities no greater than the true values of $\imm{\phi^\lambda}(A)$,
where $A$ is the path matrix of the planar network.
Specifically, these planar networks, called {\em descending star networks}
in \cite[\S 3]{CHSSkanEKL}, correspond bijectively to
unit interval orders.
This suggests the following question.
\begin{quest}
  Let $P$ be a unit interval order on $n$ elements,
  and let $D = D(P)$ be the corresponding descending star network
  as defined in \cite[\S 3]{CHSSkanEKL}.
  Do we have for all $\lambda \vdash n$
  that
  \begin{equation*}
    \phi^\lambda(\inc(P)) \geq
    \# \nTksp \bigcup_{\pi \in \mcp(D)} \ntksp \{ U \,|\,
    U \text{ a column-strict, cylindrical $\pi$-tableau of shape $\lambda$} \}?
  \end{equation*}
\end{quest}

As stated in Corollary~\ref{c:tnnimplications},
Stembridge's Conjecture~\ref{c:monimm}
is a special case of Haiman's conjecture~\cite[Conj.\,2.1]{HaimanHecke}
that $\phi_q^\lambda(\wtc wq) \in \mathbb N[q]$ for all $w$.
Furthermore, Haiman's conjecture would imply the equivalence of
seven of the eight statements of Corollary~\ref{c:tnnimplications}.
\begin{obs}
If $\phi_q^\lambda(\wtc wq) \in \mathbb N[q]$ for all $w$, then
statements $(1), (2), (4),\dotsc,(8)$ of
Corollary~\ref{c:tnnimplications} are all equivalent
to the containment
$\theta \in \spn_{\mathbb N}\{ \phi^\lambda \,|\, \lambda \vdash n \}$.
\end{obs}
\begin{proof}
  Assume that $\phi_q^\lambda(\wtc wq) \in \mathbb N[q]$ for all $w$.
  Consider
  $\theta \in \trspace n$ and expand it in the monomial trace basis as
  $\theta = \sum_\lambda a_\lambda \phi^\lambda$.
  If $a_\lambda \geq 0$ for all $\lambda$, then by our assumption we have
  \begin{equation*}
    \theta(\wtc w1) = \sum_\lambda a_\lambda \phi^\lambda(\wtc w1) \geq 0,
  \end{equation*}
  and statement (1) of Corollary~\ref{c:tnnimplications} is true.
  Now suppose that $a_\mu < 0$ for some $\mu \vdash n$ and let $w_\mu$ be
  the maximal element of the Young subgroup $\mfs \mu$.
  Then by (\ref{eq:strongestposs}) \cite[Prop.\,4.1]{HaimanHecke}
  and our assumption we have that $\theta(\wtc{w_\mu}1) < 0$.
  Since $w_\mu$ \avoidsp, statement (8) of the Corollary is false.
  \end{proof}


\ssec{Immanant identities}


Analogous to Corollary~\ref{c:traceid} are three identities
which follow from Corollary~\ref{c:immfactor}.
The third of these is known as Muir's identity.
\begin{cor}\label{c:immid}
  Let $A$ be an $n \times n$ matrix.  We have
  \begin{equation*}
    \begin{gathered}
      n\,\perm(A) = \sum_{i=1}^n \sumsb{J\\ |J|=i}
      \imm{\psi^i}(A_{J,J}) \perm(A_{\,\olj,\;\olj}),\\
      n\det(A) = \sum_{i=1}^n (-1)^{i-1} \ntksp\sumsb{J\\|J|=i}
      \imm{\psi^i}(A_{J,J}) \det(A_{\,\olj, \;\olj}),\\
      \sum_{i=0}^n (-1)^i \ntksp\sumsb{J\\|J|=i}
      \det(A_{J,J}) \perm(A_{\,\olj, \;\olj}) = 0,
    \end{gathered}
  \end{equation*}
  where the sums are over ordered set partitions of type $(i, n-i)$.
\end{cor}
\begin{proof}
  Similar to proof of Corollary~\ref{c:traceid}.
\end{proof}

\section{Applications to chromatic quasisymmetric functions}\label{s:quasi}

Shareshian and Wachs~\cite{SWachsChromQ} defined a quasisymmetric extension
$X_{G,q}$ of Stanley's chromatic symmetric function $X_G$ (\ref{eq:chrom}).
Let $G = (V,E)$ be a simple directed graph on vertices labeled $1, \dotsc, n$.
Given a proper coloring $\kappa: V \rightarrow \{1,2,\dotsc, \}$
of $G$,
define $\inv_G(\kappa)$ to be the number of pairs $(i,j) \in E$ with
$i < j$ and $\kappa(i) > \kappa(j)$.
For any composition $\alpha = (\alpha_1, \dotsc, \alpha_\ell) \vDash n$, define
\begin{equation*}
  c(G,\alpha, q) = \nTksp \sumsb{\kappa\\ \type(\kappa) = \alpha} \nTksp
  q^{\inv_G(\kappa)},
\end{equation*}
and let
\begin{equation*}
  M_\alpha = \ntksp \sum_{i_1 < \cdots < i_\ell} \ntksp
  x_{i_1}^{\alpha_1} \cdots x_{i_\ell}^{\alpha_\ell}
\end{equation*}
be the {\em monomial quasisymmetric function} indexed by $\alpha$.
Then we have the definition
\begin{equation}\label{eq:chromq}
  X_{G,q} = \sum_\kappa q^{\inv_G(\kappa)}x_{\kappa(1)} \cdots x_{\kappa(n)}
  = \sum_{\alpha \vDash n} c(G,\alpha,q) M_\alpha,
\end{equation}
where the first sum is over all proper colorings of $G$.
It is easy to see that we have $X_{G,1} = X_G$.

In the special case that $X_{G,q}$ is symmetric,
Proposition~\ref{p:everysymmfn} implies that there is an element
$g \in \mathbb Q(q) \otimes \hnq$ satisyfing
$\epsilon_q^\lambda(g) = c(G,\alpha,q)$
for every rearrangement $\alpha$ of $\lambda$.
Thus for all $\theta_q \in \trspace{n,q}$
we may define
\begin{equation}\label{eq:thetaqG}
  \theta_q(G) \defeq \theta_q(g)
  \end{equation}
to obtain a $q$-extension 
\begin{equation}\label{eq:sumXGq}
  \begin{aligned}
    X_{G,q} &= \sum_{\lambda \vdash n} \epsilon_q^\lambda(G)m_\lambda
    = \sum_{\lambda \vdash n} \eta_q^\lambda(G)f_\lambda
    = \sum_{\lambda \vdash n}\frac{(-1)^{n-\ell(\lambda)}\psi_q^\lambda(G)}{z_\lambda}p_\lambda
    = \sum_{\lambda \vdash n} \chi_q^{\lambda^\tr}(G)s_\lambda\\
    & \qquad = \sum_{\lambda \vdash n} \phi_q^\lambda(G)e_\lambda
    = \sum_{\lambda \vdash n} \gamma_q^\lambda(G)h_\lambda
    \end{aligned}
\end{equation}
of (\ref{eq:sumXG})
and a similar $q$-extension of (\ref{eq:sumomegaXG})
for $\omega X_{G,q}$.


Shareshian and Wachs showed that when $P$ is a unit interval order
labeled as in (\ref{eq:altrespect}), the function $X_{\inc(P),q}$ is in fact
symmetric.
By \cite[Cor.\,7.5]{CHSSkanEKL}, the coefficients
\begin{equation*}
  \epsilon_q^\lambda(\inc(P)),\quad
  \eta_q^\lambda(\inc(P)),\quad
  \chi_q^\lambda(\inc(P)),\quad
  \psi_q^\lambda(\inc(P)),\quad
  \phi_q^\lambda(\inc(P)),\quad
  \gamma_q^\lambda(\inc(P))
  \end{equation*}
in each standard expansion of $X_{\inc(P),q}$ are given by
$\theta_q(\inc(P)) = \smash{\theta_q(\wtc wq)}$ for the 
$312$-avoiding permutation $w$ related to $P$ as in
(\ref{eq:uioto312avoid}).
For interpretations of some of these coefficients, see
\cite[\S 5--10]{CHSSkanEKL},
\cite[\S 5--7]{SWachsChromQF}.
Shareshian and Wachs conjectured~\cite[Conj.\,5.1]{SWachsChromQ}
nonnegativity of the elementary coefficients of $X_{\inc(P),q}$.
This extends the Stanley-Stembridge conjecture~\cite[Conj.\,5.5]{StanStemIJT}
that $\phi^\lambda(\inc(P)) \geq 0$, and is equivalent to
a special case of
Haiman's conjecture~\cite[Conj.\,2.1]{HaimanHecke}
that $\phi_q^\lambda(\wtc wq) \in \mathbb N[q]$ for all $w$.
%
(For progress on this conjecture, see \cite{CMPChromatic}, \cite{HPTPermutationBases} and references there.)
\begin{conj}
  For $P$ a unit interval order labeled as in (\ref{eq:altrespect}),
  we have $\phi_q^\lambda(\inc(P)) \in \mathbb N[q]$.
  \end{conj}

Since $\sum_{\lambda \vdash n} \phi_q^\lambda = \eta_q^n$,
the sum
$\sum_{\lambda \vdash n} \phi_q^\lambda(\inc(P)) = \eta_q^n(\inc(P))$
satisfies two identities which extend (\ref{eq:etanPao})
when $P$ is a unit interval order labeled as in (\ref{eq:altrespect}).
\begin{enumerate}
\item Given an acyclic orientation $O$, define
  $\inv(O)$ to be the number of oriented edges $(j,i)$
  with $j > i$.
  We have
\begin{equation}\label{eq:monsumao}
  \sum_{\lambda \vdash n} \phi_q^\lambda(\inc(P)) =
  \sum_{O} q^{\inv(O)},
\end{equation}
  where the sum on the right is over all acyclic orientations $O$ of
  $\inc(P)$~\cite[Thm.\,5.3]{SWachsChromQ}.
\item Given a $P$-permutation $U$, define
  \begin{equation*}
    \pinv(U) = \# \{ (i,j) \,|\, j > i \text{ (as integers), } j \not >_P i,
    \text{ and $j$ appears to the left of $i$ in $U$} \}.
  \end{equation*} 
  We have
  \begin{equation}\label{eq:monsumdesfree}
    \sum_{\lambda \vdash n} \phi_q^\lambda(\inc(P)) = \sum_U q^{\pinv(U)},
  \end{equation}
  where the sum is over all $P$-descent-free $P$-permutations
  $U$~\cite[Thm.\,6.3]{SWachsChromQ}.
  \end{enumerate}
  It is possible to extend Theorem~\ref{t:etalambdaP} (2)
  to obtain a similar formula for excedance-free $P$-permutations as well.
  Given a $P$-tableau $U$, define
  \begin{equation*}
    \inv(U) = \# \{ (i,j) \,|\, j > i
    \text{ (as integers), and $j$ appears to the left of $i$ in $U$} \}.
  \end{equation*}

  
  \begin{prop}\label{p:sumphi}
    Let $P$ be an $n$-element unit interval order,
    labeled as in (\ref{eq:altrespect}) and
    corresponding by (\ref{eq:uioto312avoid})
    to the $312$-avoiding permutation $w \in \sn$.
    Then we have
    \begin{equation}\label{eq:sumphi}
        \sum_{\lambda \vdash n} \phi_q^\lambda(\inc(P))
        = \sum_{v \leq w} q^{\ell(v)}
        = \sum_U q^{\inv(U)},       
   \end{equation}
    where $U$ in the final sum
    varies over all $P$-excedance-free $P$-permutations.
  \end{prop}
  \begin{proof}
    By \cite[Cor.\,7.5]{CHSSkanEKL} and (\ref{eq:kl}) we have
    \begin{equation*}
      \sum_{\lambda \vdash n} \phi_q^\lambda(\inc(P)) =
      \eta_q^n(\inc(P)) = \eta_q^n(\wtc wq) = \eta_q^n\Big(\sum_{v \leq w} T_v\Big).
    \end{equation*}
    Since $\eta_q^n(T_v) = q^{\ell(v)}$, we have the first equality in
    (\ref{eq:sumphi}).  To see the second equality, let
    $U = v_1 \cdots v_n$ be a $P$-permutation.
    Then $U$ appears in the third sum of (\ref{eq:sumphi})
    if and only if $\exc_P(U) = 0$.
    By Corollary~\ref{c:bruhatexcfree},
    this condition is equivalent to $v \leq w$,
    and clearly we have $\ell(v) = \inv(U)$.
  \end{proof}

  Combining Proposition~\ref{p:sumphi} with
  (\ref{eq:monsumao}) and (\ref{eq:monsumdesfree}), we obtain 
    equidistribution results for the three variations of inversion statistics.
    Unfortunately, the map $w \mapsto \sigma(w^{-1})$
    from Subsection~\ref{ss:indtriv},
    which satisfies $\pexc(w) = \pdes(\sigma(w^{-1}))$ does {\em not}
    satisfy $\inv(w) = \pinv(\sigma(w^{-1}))$.
    (Neither does the map in \cite[Rmk.\,4.7]{EinarThesis} mentioned before
    Proposition~\ref{p:equidist}.)
    Furthermore,
    the statistic pairs $(\pexc, \inv)$ and $(\pdes, \pinv)$
    cannot be equidistributed on $\sn$,
    since $\inv$ and $\pinv$ are not equidistributed on $\sn$.
    This suggests the following problem for unit interval orders $P$,
    labeled as in (\ref{eq:altrespect}).
    \bp
    Find a bijection $\varphi$ from descent-free $P$-permutations to
    excedance-free $P$-permutations which satisfies
    $\pinv(U) = \inv(\varphi(U))$.
    \ep
    Just as Proposition~\ref{p:stanksources} refines (\ref{eq:etanPao}) and
    (\ref{eq:etanPdfree}),
    Shareshian and Wachs~\cite[Thm.\,5.3]{SWachsChromQ}
    refined (\ref{eq:monsumao}), (\ref{eq:monsumdesfree}) as follows.
  \begin{prop}\label{p:swksources}
    For all unit interval orders $P$
    labeled as in (\ref{eq:altrespect}),
    we have that
  \begin{equation*}
   \sumsb{\lambda\vdash n\\ \ell(\lambda) = k} \phi_q^\lambda(\inc(P))
   = \sum_U q^{\pinv(U)}
   = \sum_{O} q^{\inv(O)},
  \end{equation*}
  where the second and third sum are over descent-free $P$-permutations $U$
  having $k$ $P$-records, and
  acyclic orientations $O$ of $\inc(P)$ having $k$ sources.
\end{prop}

  It would be interesting to similarly refine
  Proposition~\ref{p:sumphi}.
  \bp
  Let $P$ be an $n$-element
  unit interval order labeled as in (\ref{eq:altrespect}),
  and let $w(P)$ be the corresponding $312$-avoiding permutation
  as in (\ref{eq:uioto312avoid}).  Find functions
  $\delta_1$, $\delta_2$ so that
  \begin{equation*}
    \sumsb{\lambda \vdash n\\ \ell(\lambda) = k} \ntksp \phi_q^\lambda(\inc(P))
    = \nTksp \sumsb{v \leq w\\ \delta_1(v,w) = k} \nTksp q^{\ell(v)}
    = \nTksp \sumsb{\pexc(U) = 0\\ \delta_2(U) = k} \nTksp q^{\inv(U)},
  \end{equation*}
  where $U$ in the final sum is a $P$-permutation.
\ep  
Finally, it would be interesting to combinatorially settle
the Shareshian-Wachs conjecture~\cite[Conj.\,5.1]{SWachsChromQ} that
$\phi_q^\lambda(\inc(P)) \in \mathbb N[q]$ as follows.
\bp
For each $n$-element unit interval order $P$
labeled as in (\ref{eq:altrespect}), and
each partition $\lambda \vdash n$,
define subsets $\mathcal S_1(\lambda)$,
$\mathcal S_2(\lambda)$,
$\mathcal S_3(\lambda)$,
$\mathcal S_4(\lambda)$
of acyclic orientations of $\inc(P)$,
$P$-descent-free $P$-permutations,
$\{ v \in \sn \,|\, v \leq w(P) \}$,
and $P$-excedance free $P$-permutations, respectively, so that
\begin{equation*}
  \phi_q^\lambda(\inc(P))
  = \ntksp \sum_{O \in \mathcal S_1(\lambda)} \ntksp q^{\inv(O)}
  = \ntksp \sum_{U \in \mathcal S_2(\lambda)} \ntksp q^{\pinv(U)}
  = \ntksp \sum_{v \in \mathcal S_3(\lambda)} \ntksp q^{\ell(v)}
  = \ntksp \sum_{U \in \mathcal S_4(\lambda)} \ntksp q^{\inv(U)}.
  \end{equation*}
\ep



\section{Acknowledgements}
The author is grateful to an anonymous referee for numerous corrections
and suggestions which helped improve the article significantly.

\end{document}